\documentclass[11pt]{amsart}
\usepackage{amsmath,amsthm,amscd,amssymb}
\usepackage{geometry}
\usepackage{setspace}

\usepackage{tikz}
\usetikzlibrary{matrix,arrows,decorations.pathmorphing}

\usepackage{natbib}

\usepackage{pb-diagram}
\usepackage{graphicx}
\usepackage{caption}
\usepackage{subcaption}

\usepackage{epstopdf}

\usepackage{amssymb}

\newtheorem{thm}{Theorem}[section]
\newtheorem{lem}[thm]{Lemma}
\newtheorem{prop}[thm]{Proposition}
\newtheorem{cor}[thm]{Corollary}

\theoremstyle{definition}

\newtheorem{defn}[thm]{Definition}

\newtheorem{ques}[thm]{Question}

\numberwithin{equation}{section}
\numberwithin{figure}{section}

\newcommand{\Z}{\mathbb{Z}}

\newcommand{\C}{\mathcal{C}}

\newcommand{\Q}{\mathbb{Q}}
\newcommand{\R}{\mathcal{R}}

\newcommand{\Real}{\mathbb{R}}
\newcommand{\Comp}{\mathbb{C}}
\renewcommand{\P}{{\mathbb{P}}}

\usepackage{pstricks}
\usepackage{pst-node}

\newcommand{\blf}{\mathcal{B} \ell}
\newcommand{\cA}{\mathcal{A}}
\newcommand{\F}{\mathcal{F}}

\numberwithin{figure}{section}

\title[The Role of Link Concordance in Knot Concordance]{The Role of Link Concordance in Knot Concordance}

\author{Diego Vela}
\address{}
\email{supotuco@gmail.com}

\begin{document}
\date{\today}

\begin{abstract}

Satellite constructions on a knot can be thought of as taking some strands of a knot and then tying in another knot.  Using satellite constructions one can construct many distinct isotopy classes of knots.  Pushing this further one can construct distinct concordance classes of knots which preserve some algebraic invariants.  Infection is a generalization of satellite operations which has been previously studied.  An infection by a string link can be thought of as grabbing a knot at multiple locations and then tying in a link.  Cochran, Friedl and Teichner showed that any algebraically slice knot is the result of infecting a slice knot by a string link \citep{CFT09}.  In this paper we use the infection construction to show that there exist knots which arise from infections by $n$-component string links that cannot be obtained by infecting along an $(n-1)$-component string links.  

\end{abstract}

\maketitle

\section{Introduction}\label{sec:Introduction}

A knot is an embedding of $S^1$ into $S^3$ and we denote it by $K: S^1 \hookrightarrow S^3$, and a link $L$ is an embedding $L: \coprod S^1 \hookrightarrow S^3$.  Knots and links are equivalent if they are ambient isotopic.  Two knots $J$ and $K$ are ambient isotopic exists map $\Phi: S^3 \times I \rightarrow S^3$ such that $\Phi(x,t)$ is a homeomorphism for all $t$, $\Phi(x,0)$ is the identity, and $\Phi(K,1) = J$.  We wish to study the set of algebraically slice knots, which we define below.  In 1969 Levine defined the algebraic concordance group \citep{JL69}.  He also defined a map $\phi$ from the set of knot concordance classes $\C$ to the set of algebraic concordance classes $\cA \C \cong \Z^{\infty} \oplus (\Z/2) ^{\infty} \oplus (\Z/4)^{\infty}$.  We study the kernel of this map.

\begin{defn}
A knot $K$ is algebraically slice if $\phi(K) = 0$, let $\cA \mathcal{S}$ denote $ kernel(\phi)$.  
\end{defn}

In higher dimensions the set of slice knots coincides with the set of algebraically slice knots.  Casson and Gordon showed in \citep{CG76} that the set of algebraically slice knots is a proper subset of the set of slice knots.  We seek to better understand the difference between these two sets.  

In Section \ref{sec:doubOp} we use a modified version of the Cochran-Orr-Teichnrer filtration of $\C$, denoted $\F_n^P$, to prove our main results.  Here $P$ references the modification.  These generalized filtrations are by Cochran-Harvey-Leidy, and Burke.  We denote the original Cochran-Orr-Teichnrer Filtration by $\F_n$.  A controlled way to construct knots $K$ with the property that $K$ is an element of $\F_n^P$ is through satellite constructions; see Section \ref{sec:doubOp}, \citep{CHL11}, \citep{JB14} for details.  We think of a satellite construction as grabbing some strands of a knot $R$ and tying in another knot $K$.  Infection is a generalization of the satellite constructions.  Similarly, we can think of infection as grabbing a knot at multiple locations and tying in a link $L$.  There are three pieces of data in an infection, $R$, $T$ and $L$ and we denote the result as $R_T(L)$; again see Section \ref{sec:doubOp} or \citep{JB14} for details.

Cochran, Friedl and Teichner showed that for any algebraically slice knot $K \in \cA \mathcal{S}$ there exists a link $L$ and a ribbon knot $\R$ such that as concordance classes $[K]$ and $[\R_T(K)]$ are the same \citep[Proposition 1.7]{CFT09}.  In other words, up to the equivalence relation of concordance all knots are obtained by infecting a ribbon knot by a link.  The following questions arise naturally.  

\begin{ques}\label{ques:BigQ}
What is the dependence on the set of links? 
\end{ques}

\begin{ques}\label{ques:LittleQ}
Is there number $\ell$ such that every algebraically slice knot is obtained by an infection by an $\ell$-component string link?
\end{ques}

Our main result, Theorem \ref{thm:nonTriviality}, gives a partial answer to Question \ref{ques:LittleQ}.  Using Theorem \ref{thm:nonTriviality}, we can construct algebraically slice knots which arise from infecting a ribbon knot $\R$ along an $\ell$-component string link $L$ such that $\R_T(L)$ is not concordant to any infection of the form $\R'_{T'}(L')$ where $L'$ is a string link with fewer components.  There are some restrictions on $L$ and $L'$ which prohibits us from a complete answer.

A reasonable philosophy is that one must understand the knot concordance set before one can understand the link concordance set.  Using infection techniques and algebraic techniques developed by Cochran, Harvey, and Leidy in \citep{CHL11} we have given evidence that we must simultaneously understand the set of link concordance classes and knot concordance classes.  The operators developed in \citep{CHL11} are known as doubling operators and are denoted $\R_{\eta}: \C \rightarrow \C$.  These operators are functions on the set of knot concordance classes.  The goal of \citep{CHL11} was to describe a primary decomposition of the concordance classes with respect to different types of doubling operators.

In \citep{JB14} Burke has a theorem which we think of as a triviality result.  His result essentially states that one can construct a modified Cochran-Orr-Teichnrer filtration $\F_n^P$.  This filtration has a nice property that if the higher order Alexander modules of a knot $K$, which is in $\F_n$, do not ``match" $P$ then $K$ is in $\F_{n+1}^P$.  The modified filtration does not guarantee that if a knot has the appropriate higher order Alexander modules that it is nontrivial in $\F_{n}^P/\F_{n+1}^P$.  Our main result fills this gap by constructing examples which are nontrivial in this successive quotient.  A note is that Burke does construct infections by 2-component links in  \citep{JB14}; our results are for $n$-component links.

The paper is organized in the following way.  Section \ref{sec:prelim} contains some basic definitions of concordance and infection.  Section \ref{sec:strongIrred} develops the algebraic tools that we use to obstruct concordances.  In Section \ref{sec:linkPoly} we construct our examples.  Section \ref{sec:filtration} reviews the $n$-solvable filtration and defines the modified version, the $(n,*)$-filtration.  In Section \ref{sec:doubOp} we prove nontriviality of our examples.  In Section  \ref{sec:goodAModule} we prove some important properties of the Blanchfield form for abstract links with given torsion Alexander polynomials(defined in Section \ref{sec:prelim}).  These properties may be sufficient to show abstractly that other links can be substituted into the proof of triviality/nontriviality but one would need to find a link that realizes the properties.

\section{Preliminary Material}\label{sec:prelim}

\subsection{Overview}

We give a brief overview of some definitions and motivation.  For a more complete treatment see \citep{DR90}.  Knot and link theory is related to the smoothing theory of 4-manifolds as follows.  Let $f:\Comp \rightarrow \Comp^2$ be a smooth map whose image is singular at the origin.  Let $U$ be a neighborhood of the singularity.  One can intersect $U$ with the image of $f$ and obtain a link $L$.  We can smooth out $f$ to be an embedding if the link bounds disjoint embedded disks.  From this we have the following definitions.

\begin{defn}
A knot $K$ is {\em smoothly slice} if there exists a smoothly embedded disk $D^2 \hookrightarrow B^4$ such that $\partial(B^4,D^2) = (S^3,K)$.  
\end{defn}

\begin{defn}
A link $L$ is {\em smoothly slice} if each component is slice and all the slice disks can be taken to be disjoint.  
\end{defn}  

\begin{defn}
A slice link $L$ is {\em ribbon} if you can take the slice disks to only have index 0 and index 1 critical points for a Morse function on $B^4$.
\end{defn}

For knots, below we define a natural equivalence relation known as {\em concordance}.  For a knot $J$, let $rJ$ denote the knot $J$ with the opposite orientation.  For a knot $J$, let $\overline{J}$ denote the mirror of $J$.  More specifically we can think of $J$ as a subset of $\mathbb{R}^3$, through stereographic projection, and then $\overline{J}$ is the image of $J$ when we reflect through a plane in $\mathbb{R}^3$.  

\begin{defn}
Two knots $J,K$ are {\em concordant} if $K\# r \overline{J}$ is slice.  We denote the concordance class by $[K]$ or $[J]$.
\end{defn}

We denote the set of concordance classes by $\C$.  One can show that two knots are concordant if and only if they cobound an annulus.  More precisely $J,K$ are concordant if there exists a smooth embedding $A: S^1 \times I \hookrightarrow S^3 \times I$ such that one end of the annulus is $J$ and the other is $K$.  One can define a notion of concordance between two links using this definition.  The problem of studying slice knots is similar to studying knot concordance.

\subsection{String Links}

We use string links throughout this paper, defined as follows.

\begin{defn}\label{defn:stringLink}
Let $x_1,\dots,x_n$ be marked points in $D^2$.  A {\em string link} of $n$-components is a smooth embedding $L: \{1,\dots,n\} \times I \hookrightarrow D^2 \times I$ such that $L(i,0) = (x_i , 0)$ and $L(i,1) = (x_i , 1)$.  
\end{defn}

String links naturally close up to produce links.  The closure can be thought of as gluing the top to the bottom.  It is defined as follows.

\begin{defn}\label{defn:stdClosure}
Given a string link $L$, the {\em standard closure} of $L$ is defined to be the link $L'$ which is obtained by first identifying $D^2 \times \{0\}$ with $D^2 \times \{1\}$ using $-Id$, which yields a solid torus $D^2 \times S^1$.  Fix a point $p$ on the boundary of $D^2$.  Observe that the boundary of $D^2 \times I/\sim$ is a torus.  Inside this torus we have $p\times I$ mapping to $p \times S^1$.  Attach a solid torus to the boundary such that the meridian of the solid torus maps to $p \times S^1$ to obtain $S^3$.
\end{defn}

\begin{defn}\label{defn:knotType}
To each component $L_i$ of a string link $L$ there is a knot $K$ that comes from the standard closure.  We call $K$ the {\em knot type} of the component $L_i$.
\end{defn}

One thing to observe is that if you remove an open neighborhood of the string link with $g$ components from $D^2 \times I$ you have a surface $\Sigma_g$ with genus $g$ on the boundary.  There are some natural curves on the boundary, namely the meridians and the longitudes.  For Definition \ref{defn:meridian} keep in mind that the boundary of $D^2 \times I \backslash L_i$ is a homology solid torus.

\begin{defn}\label{defn:meridian}
For a given string link $L$, a {\em meridian} $\mu_i$ associated to the $i$th component $L_i$ is a curve $\mu_i$ that embeds into $ \partial(D^2 \times I - \nu(L_i))$, with the homology class $[\mu_i]$ generates $H_1(D^2 \times I \backslash L_i; \Z)$ and $\mu_i$ is isotopic to a curve in $\nu(L_i) \backslash L_i$ in $D^2 \times I \backslash L$.    
\end{defn}

\begin{defn}\label{defn:longitude}
For a given string link $L$, a {\em longitude} $\lambda_i$ in $\partial (D^2 \times I \backslash \nu(L))$ is a curve such that $\lambda_i$ which embeds into $\partial(D^2 \times I \backslash \nu(L_i))$ and the geometric intersection of $\lambda_i$ with $\mu_i$ is $1$ and the geometric intersection of $\lambda_i$ with $\mu_j$ is $0$ otherwise.
\end{defn}

\begin{defn}\label{defn:preferredLongitude}
For each component the {\em preferred longitude} is a longitude $\lambda_i$ such that $[\lambda_i] $ is $0$ in $H_1(D^2 \times I \backslash \nu(L))$ under inclusion.
\end{defn}

There is also a formulation of concordance for string links as follows.

\begin{defn}\label{defn:stringLinkConcordance}
We say that two string links $L_1,L_2$ are {\em concordant} if there exists a smooth embedding $A:  (\{1,\dots,n\} \times I) \times I \hookrightarrow D^2\times I \times I$ such that $A(i,I,0) = L_1(i,I)$, $A(i,I,1) = L_2(i,I)$, $A(i,0,t) = (x_i , 0 , t)$ and $A(i,1,t) = (x_i , 1 , t)$.  
\end{defn}

The set of string links with $n$ components forms a monoid with the stacking as the operation.  Using the equivalence relation of concordance the set of string links with $n$ components becomes a group.  We denote the string link concordance groups by $\C_n$ where $n$ is the number of components.  Sometimes we omit the index from $\C_1$, since $\C_1$ is naturally isomorphic to the knot concordance group $\C$.

We also use the zero framed surgery of a string link, defined as follows.

\begin{defn}\label{def:zeroSurgery}
The {\em zero framed surgery} $M_L$ of a string link $L$ is obtained by taking the string link complement and then attaching a copy of $D^2$ along each preferred longitude and then attaching a $B^3$ along the remaining boundary.  
\end{defn}

The exterior of the string link is a surface $\Sigma$.  The $\lambda_i$ are embedded simple closed curves in $\Sigma$ which are also independent in homology.  Therefore, each time a $D^2 \times I$ is attached along a different $\lambda_i$ the genus is reduced by one.  This can be checked by an Euler characteristic argument.  After attaching $n$ disks all that remains is a surface with positive Euler characteristic which is $S^2 = \partial B^3$.  One can think of the zero surgery process as attaching a handlebody to the exterior of the string link.

For the rest of the paper we assume that the components of all links and string links have pairwise linking number 0.  It is a necessary condition for a link to be slice, therefore it is not a strong assumption.  The following is the definition of infecting a knot by a string link.  

\begin{defn}\label{def:infectRbySL}
Let $R$ be a knot, $L$ be a string link, $E(L)$ denote the exterior of $L$, and $T$ the exterior of the trivial string link which is embedded in the $S^3 \backslash R$ such that the preferred longitudes bound embedded disks.  We remove the trivial string link and observe that we have a knot $R \hookrightarrow S^3 - \nu(T)$.  Let $\lambda_i$ and $\mu_i$ denote the preferred longitude and the meridian of $T$.  {\em The infection} on $R$ by a string link $L$, denoted $R_T(L)$, is obtained by identifying $\lambda_i$ with the preferred longitude of $L$ in $E(L)$ and $\mu_i$ with the meridian of $L$ $E(L)$ and then identifying the rest of the boundary.  Since $\Sigma = \partial (S^3 \backslash \nu(T))$ is a surface and a $K(\pi_1(\Sigma),1)$ these conditions are sufficient to define a map on the boundary.  
\end{defn}

One thing to note about string links is that each component is marked.  The $i$-th component of a string link $L$ is denoted $L_i$ and is defined to be the component such that $L(i,0) =  (x_i,0)$.  For example, let $L$ be a 2-component string link such that the standard closure of $L_1$ is a trefoil and the standard closure of $L_2$ is an unknot.  Then let $L'$ be a 2-component string link such that the standard closure of $L'_1$ is an unknot and the standard closure of $L'_2$ is a trefoil.  The standard closures of $L$ and $L'$ these links are equivalent.  The string links $L,L'$ are distinct.  This marking is embedded into the process of infection.  It is inherent into the choice of embedding of $T$.  The pictures in Figure \ref{fig:InfectionProcess} give a visual description of infection by a string link.

\begin{figure}[h!]
\includegraphics[width = 6cm] {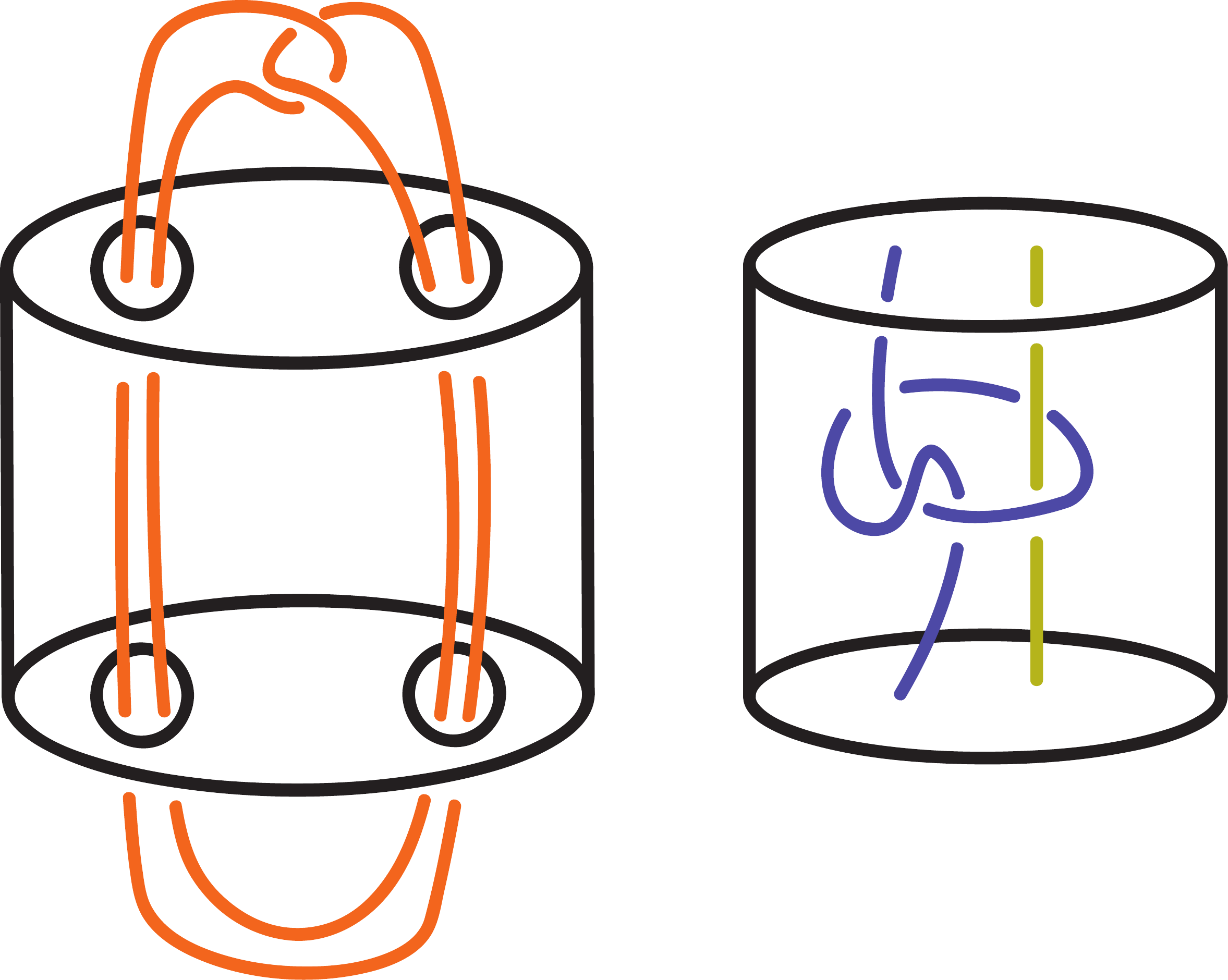}
\put(-175,120){$\R$}
\put(-185,60){$T$}
\put(-50,110){$L$}
\includegraphics[width = 3cm]{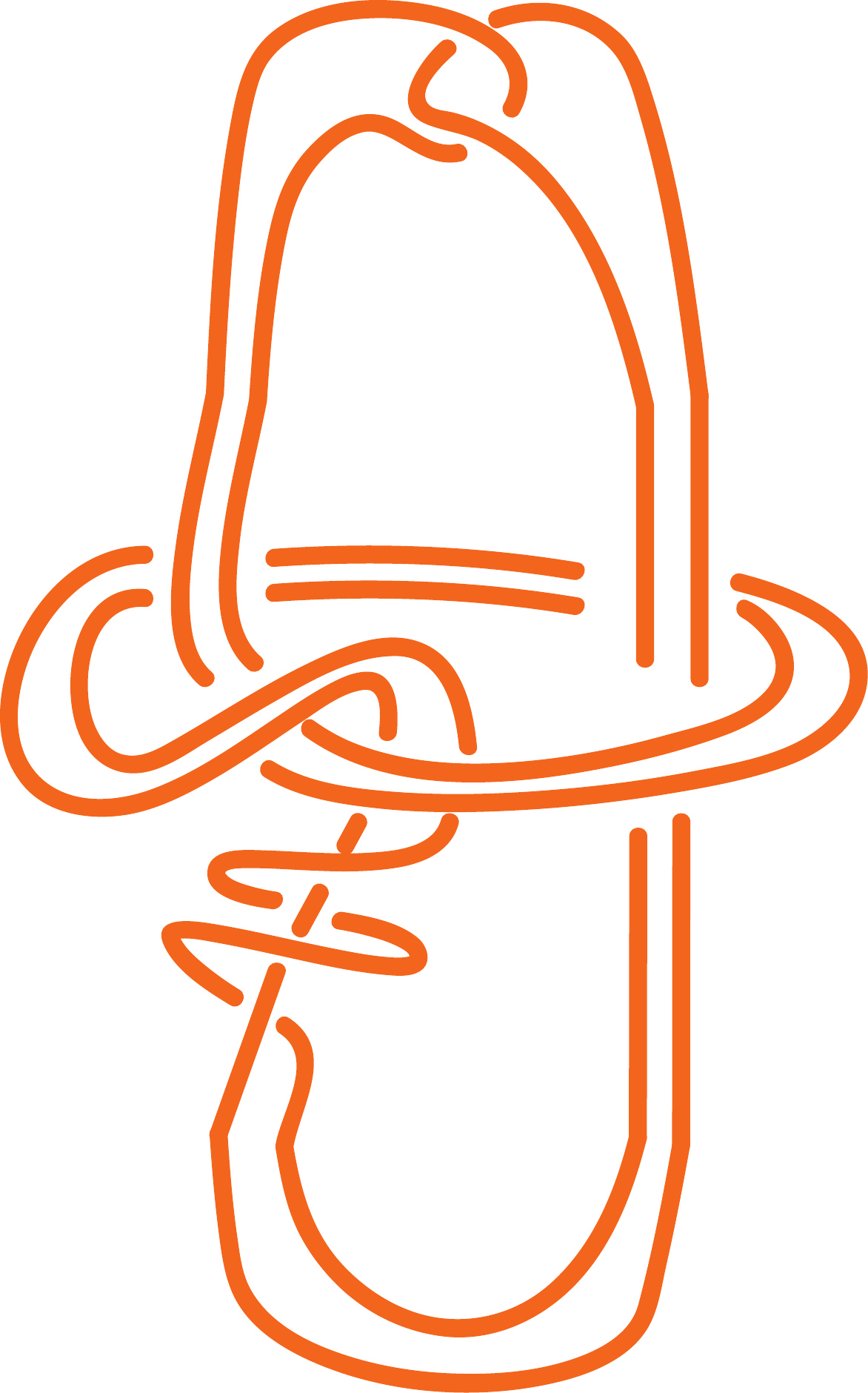}
\put(-10,100){$\R_{T}(L)$}
\caption{Infection Process}
\label{fig:InfectionProcess}
\end{figure}


Moreover a string link can {\em infect a string link} as follows.  Let $L,J$ be string links and let $T$ be a copy of the exterior of a trivial string link embedded in $D^2 \times I - L$ such that the preferred longitudes and meridians bound embedded disks in $D^2\times I$.  The infection on $L$ by $J$, denoted $L_T(J)$, is obtained by following the process in Definition \ref{def:infectRbySL}.

When $T$ is understood we will suppress it from the notation and just write $R(L)$ or $L(J)$.  Lemma \ref{lem:wellDefined} is well known but we include it for completeness.

\begin{figure}[h!]
\includegraphics[width = 6 cm]{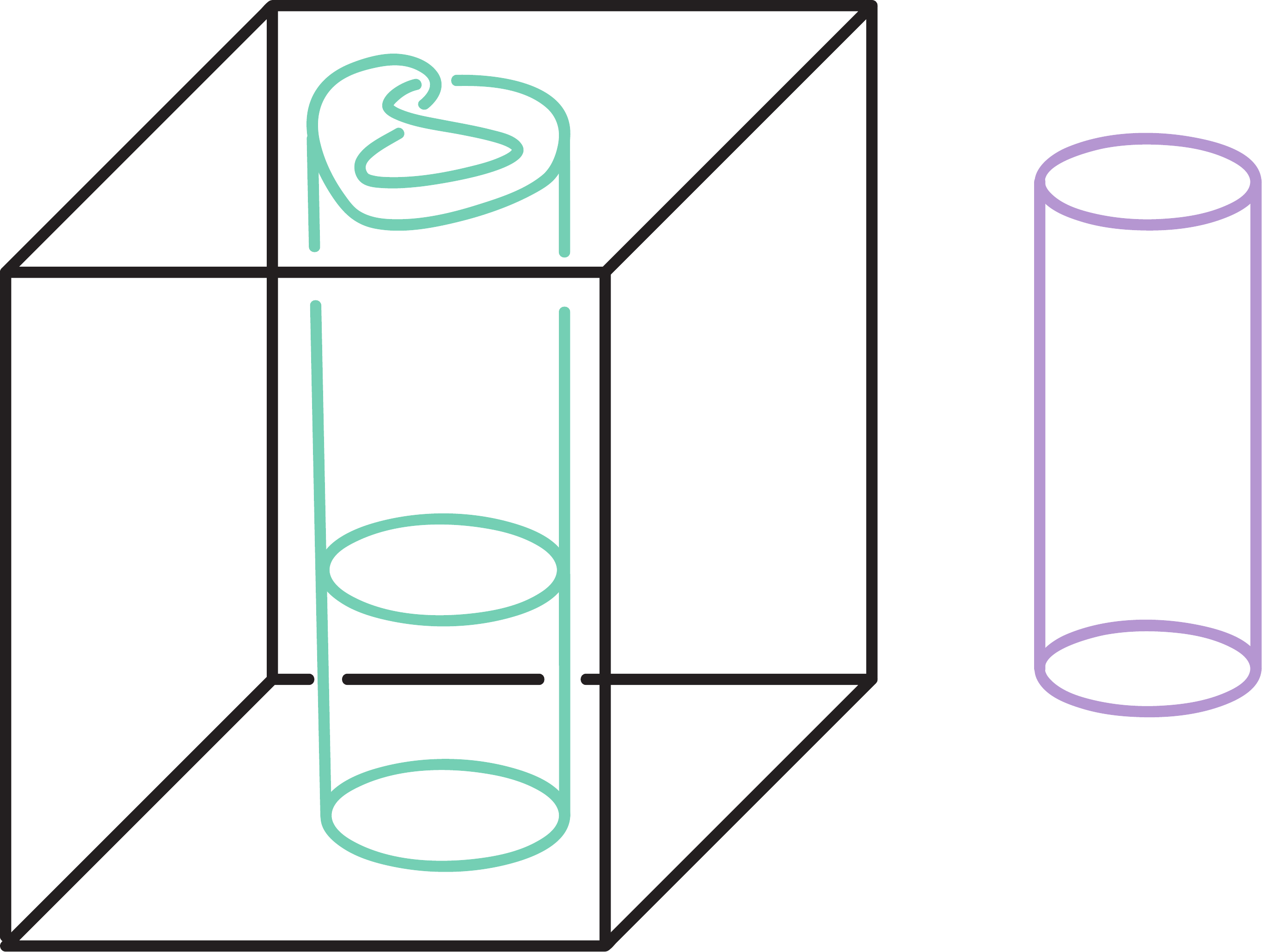}
\put(-240,90){$(S^3 \times I) - C$}
\put(10,90){$L \subset (S^3 - \eta) \times I$}
\caption{Concordance}
\label{fig:Concordance}
\end{figure}

\begin{lem}\label{lem:wellDefined} 
Infection is well defined on concordance classes of string links.
\end{lem}

\begin{proof} We show that if $c,d$ are concordant as string links then $L(c)$ and $L(d)$ are concordant.  We focus on $c,d$ being 1-component string links because the proof for arbitrary components is the same.  For the infections $L(-)$ we have a simple closed curve $\eta \subset D^2 \times I$ where $\eta \cap L = \emptyset$.  Take a tubular neighborhood of $\eta$ which is a solid torus, call it $\nu(\eta)$.  The standard longitude of $\nu(\eta)$ bounds a disk in $D^2 \times I$ which has a product neighborhood and intersects $L$ transversely, call it $D_\eta$.  Take the product neighborhood such that each intersection of $D_\eta$ with $L$ goes from one side to the other.  The plan is to glue in $D_\eta \times I$ into the concordance.  Since $c,d$ are concordant there exists an embedding of  $I \times I$ into $D^2 \times I$ such that $(a,0) = c$ and $(a,1) = d$ and $I \times I $ has a product neighborhood.  So take the product neighborhood of the embedded $I \times I$ to get an embedding of $D^2 \times I \times I$.  Remove the product neighborhood and replace it with $D_\eta \times I \times I$, call the resulting space $Y$.  Next we build the concordance.  Remove $D_\eta \times I \cup \nu(\eta)$ from $ D^2 \times I$ and call the resulting space $X'$.  Let $X = X' \times I$ and glue in $Y$ such that $\{x\times \{0\}\} = L(c)$ and $\{x\times \{1\}\} = L(d)$ to get the concordance.  The reason it is the same for string links with more components is that the embeddings of $I \times I$ are disjoint\end{proof}
\subsection{Algebraic Invariants}

Throughout the paper we use a definition of the Alexander polynomial which differs from the classical one.  Typically the Alexander polynomial is defined as an annihilator of the Alexander module. For a link, the Alexander module may have a torsion free part.  Before we define the actual polynomial we list some algebraic definitions and the definition of the torsion Alexander module. 

\begin{defn}\label{def:torsionAlexanderModule}
Let $L$ be a link and $\Gamma = \pi_1/[\pi_1,\pi_1]$.  The {\em torsion Alexander module}, $T\cA$, of a link $L$, is the torsion submodule of the Alexander module $H_1(S^3 \backslash L, \Z[\Gamma])$, that is, the submodule $$T\cA(L) = \{x \in H_1(S^3 \backslash L, \Z[\Gamma])  \mid \exists p \neq 0 \in \Z[\Gamma]\mbox{ such that }px = 0\}.$$
\end{defn}

The group $\Gamma$ in Definition \ref{def:torsionAlexanderModule} is isomorphic to $\Z^n$, where $n$ is the number of components in $L$, since the components are pairwise linking number 0.  For a knot the torsion Alexander module is equal to the classical Alexander module.  For a link it is possible for the torsion Alexander module to be trivial, but the module to be nontrivial.  Now we repeat some algebraic definitions; more details can be found in \citep[Ch. 3]{JH12}.  We are assuming that we have a module $M$ and a presentation for the module $M$, that is an exact sequence $R^p \rightarrow R^q \rightarrow M \rightarrow 0$.  

\begin{defn}\label{defn:elementaryIdeal}\citep[Ch. 3]{JH12}
The {\em k-th elementary ideal}, $E_k(M)$ is the ideal generated by the $(q-k)\times(q-k)$ sub-determinants of the matrix presenting $M$.  
\end{defn}

\begin{defn}\citep[Ch. 3]{JH12}\label{defn:divisorialHull}
Given an ideal $I$ in a ring $R$, the {\em divisorial hull} $\tilde{I}$ is the intersection of the principal ideals of $R$ which contain $I$.  This is a principal ideal.
\end{defn}

Since $\Z[t_1^{\pm1},\dots,t_n^{\pm1}]$ is a unique factorization domain and $T\cA(L)$ is finitely generated over $\Z[t_1^{\pm1},\dots,t_n^{\pm1}]$, we have a generator $\Delta_k( T\cA(L) )$ of $\widetilde{E}_k(T\cA(L))$.  Note that for $\Delta_k$ in Hillman's notation $k$ is a natural number.  

\begin{defn}\label{def:torsionAlexanderPolynomial}  Let $\widetilde{E}_0(T\cA(L))$ be the divisorial hull of the 0-th elementary ideal.  We define the {\em torsion Alexander polynomial} $\Delta_L$ of a link $L$ to be a generator of $\widetilde{E}_0(T\cA(L))$.  The {\em torsion Alexander polynomial} $\Delta_{L_2}$ of a string link $L_2$ is the torsion Alexander polynomial of the standard closure of $L_2$
\end{defn}

One thing to note is that the torsion Alexander polynomial agrees with the classical Alexander polynomial for a knot.  See \citep[Ch. 3]{JH12} for details.

\section{Detecting Strong Irreducibility}\label{sec:strongIrred}

\subsection{Strong Irreducibility}
Some of the items in this section are well known; see \citep{hartshorneAlgGeometry} and \citep{SHA13} for details.  The original results in this section are related to {\em strongly irreducible} and {\em strongly coprime}.

We begin with the definition of strongly coprime and then we focus on some specific cases that will be used for applications.  Let $R$ be a commutative ring with unity and consider the polynomial algebras $R[x_1,\dots,x_n]$ and $R[x_1^{\pm1},\dots,x_n^{\pm1}]$.

\begin{defn}\label{def:strongCp}
Suppose that $p(x_1,\dots,x_n)$ and $q(x_1,\dots,x_n)$ are in $R[x_1,\dots,x_n]$.  We say that $p$ and $q$ are {\em strongly coprime}, denoted $\widetilde{(p,q)} = 1$, if for any finitely generated free abelian group $A$ and for each pair of linearly independent sets $\{a_1,\dots,a_n\}$, $\{b_1,\dots,b_n\} \subset A$ we have that $p(a_1,\dots,a_n)$ is coprime to $q(b_1,\dots,b_n)$ over the group ring $R[A]$.  When $q$ is a polynomial with fewer variables we would only use a subset of $\{b_i\}$ that has as many vectors as $q$ has variables.  If two polynomials are not strongly coprime then we say they are {\em isogenous}.
\end{defn}

\begin{defn}\label{def:strongIr}
A polynomial $p(x_1,\dots,x_n)$ in $R[x_1,\dots,x_n]$ is {\em strongly irreducible} if for any finitely generated free abelian group $A$ and a linearly independent set $\{a_1,\dots,a_n\}$ the evaluation $p(a_1,\dots,a_n)$ is irreducible in $R[A]$.  
\end{defn}

For the rest of the paper we focus on the cases $R = \Z,\Q,\overline{\Q},\Real,\Comp$, where $\overline{\Q}$ is the algebraic completion of $\Q$.  We also choose to work with homogeneous polynomials in order to use some results from algebraic geometry to get more information about these polynomials.  Definition \ref{def:strongCp} is similar to Definition 4.1 in \citep{JB14} but differs slightly because Burke does not require that $\{b_i\}$ is a linearly independent set.  The applications for the definitions end up being the same.  In Section \ref{sec:filtration} we use the polynomials to create a localization set.  While we have different definitions of being strongly coprime one can use the linear dependence in Definition 4.1 from \citep{JB14} to show that the localization sets in \ref{sec:filtration} are the same sets.

Showing two polynomials are strongly coprime is a difficult problem so we focus on polynomials being strongly irreducible.  A polynomial $p$ in $R[x_1^{\pm1},\dots,x_n^{\pm1}]$ is strongly irreducible if for any nonzero integral choice of $\{t_i\}_{i=1}^{n}$ the polynomial $p(x_1^{t_1},\dots,x_n^{t_n})$ is irreducible.  Lemma \ref{lem:iredThenLaurentIred} allows us to work over $R[x_1,\dots,x_n]$ instead and we will apply it implicitly.  We remind the reader that we are focusing on the rings $R = \Z,\Q,\overline{\Q},\mathbb{R},\Comp$ so that the polynomial rings $R[x_1,\dots,x_n]$ are unique factorization domains.  

\begin{lem}\label{lem:iredThenLaurentIred}
If $p(x_1,\dots,x_n)$ is irreducible in $R[x_1,\dots,x_n]$ and if $p(x_1,\dots,x_n) \neq kx_j$ for any $j$ and any $k$ in $R$ then $p(x_1,\dots,x_n)$ is irreducible in $R[x_1^{\pm1},\dots,x_n^{\pm1}]$.
\end{lem}

\begin{proof} Let $p(x_1,\dots,x_n) \in R[x_1,\dots,x_n]$ such that $p$ is irreducible and assume to the contrary that $p$ is not Laurent irreducible and not a unit in $R[x_1^{\pm1},\dots,x_n^{\pm1}]$.  In other words $p$ is not equal to $kx_j$ where $k$ is a unit in $R$.  Then we can factor $p(x_1,\dots,x_n) = f g$ where $f,g \in R[x_1^{\pm1},\dots,x_n^{\pm1}]$ such that neither $f$ nor $g$ is a unit.  Let $r_i,s_i$ be the absolute values of the lowest degrees of $x_i$ in $f$ and $g$ respectively.  Then $p(x_1,\dots,x_n) x_1^{r_1+s_1} \cdots x_n^{r_n+s_n} = f x_1^{r_1} \cdots x_n^{r_n} g x_1^{s_1} \cdots x_n^{s_n}$.  Let $F = f x_1^{r_1} \dots x_n^{r_n}$ and observe that $F \in R[x_1,\dots,x_n]$.  Similarly $G = g  x_1^{s_1} \dots x_n^{s_n}$ in $R[x_1,\dots,x_n]$.  Then $p(x_1,\dots,x_n) x_1^{r_1+s_1} \dots x_n^{r_n+s_n} = FG$ is an equation in $R[x_1,\dots,x_n]$.  Since $R[x_1,\dots,x_n]$ is a unique factorization domain, if $R$ is, we have that $$p(x_1,\dots,x_n) x_1^{r_1+s_1} \dots x_n^{r_n+s_n} = FG = F'G' x_1^{r_1+s_1} \dots x_n^{r_n+s_n}.$$  Since $R[x_1,\dots,x_n]$ is an integral domain we can cancel $x_i^{r_i+s_i}$ from both sides.  Since $F'$ and $G'$ are polynomials in $R[x_1,\dots,x_n]$ and since $p$ is irreducible, then either $F'$ or $G'$ is a unit in $R[x_1,\dots,x_n]$.  Since the only units in $R[x_1,\dots,x_n]$ are constants this implies that either $F' = k$ or $G' = k$.  Without loss of generality suppose $F' = k$.  Observing that $F' = fx_1^{t_1}\cdots x_n^{t_n}$ we reach a contradiction because $f = F' x_1^{-t_1}\cdots x_n^{-t_n}$ is a unit in $R[x_1^{\pm1},\dots,x_n^{\pm1}]$\end{proof}

It is sometimes easier to work with polynomials in $R[x_1,\dots,x_n]$, which leads us to the following definition.

\begin{defn}\label{def:strIred2}
A polynomial $p$ in $R[x_1,\dots,x_n]$ is strongly irreducible if for any nonnegative choice of $\{t_i\}$ the polynomial $p(x_1^{t_1},\dots,x_n^{t_n})$ is irreducible.
\end{defn}
Lemma \ref{lem:iredThenLaurentIred} tells us that if a polynomial is strongly irreducible over $R[x_1,\dots,x_n]$ then it is strongly irreducible.  {\em For the rest of this section by strongly irreducible we will mean strongly irreducible in $R[x_1,\dots,x_n]$}.  Proposition \ref{prop:irreducibleIFFForAllK} is a tool to easily determine when a polynomial is strongly irreducible.  We then prove a result to show that no information is lost in working with homogeneous polynomials.  Most of our results relate to strongly irreducible polynomials.

\begin{prop}\label{prop:irreducibleIFFForAllK}
The polynomial $p(x_1,\dots,x_n)$ in $R[x_1,\dots,x_n]$ is strongly irreducible if and only if for any $k > 0,$ $p(x_1^k,\dots,x_n^k)$ is irreducible.
\end{prop}

\begin{proof} Assume that $p$ is strongly irreducible.  Letting $t_i = k$ for each $i$, we see that $p(x_1^k,\dots,x_n^k)$ is irreducible.

To prove the other direction we show the contrapositive.  Assume that $p(x_1^{t_1},\dots,x_n^{t_n})$ factors for some choice of $\{t_i\}_{i=1}^{n}$.  We fist establish that it is sufficient to consider $t_i > 0$.  If $t_i < 0 $ then for we make the substitution by setting $x_i$ equal to $y_i$ if $t_i > 0$ and $y_i^{-1}$ if $t_i < 0$.  Then $p(x_1^{t_1},\dots,x_n^{t_n}) = p(y_1^{\mid t_1 \mid}, \dots, y_n^{\mid t_n \mid})$ in $R[y_1,\dots,y_n]$.  Since we are assuming that $p(x_1^{t_1},\dots,x_n^{t_n})$ factors over $R[x_1^{\pm1},\dots,x_n^{\pm1}]$, it follows that $p(y_1^{\mid t_1 \mid}, \dots, y_n^{\mid t_n \mid})$ factors over $R[y_1^{\mid t_1 \mid}, \dots, y_n^{\mid t_n \mid}]$.  By Lemma \ref{lem:iredThenLaurentIred} it factors over $R[y_1,\dots,y_n]$.  Therefore it suffices to assume that $t_i > 0$ for all $i$.  

Then $p(x_1^{t_1},\dots,x_n^{t_n}) = f(x_1,\dots,x_n) g(x_1,\dots,x_n)$, for some $f$ and $g$ which are not constant.  Let $L = lcm(t_1,\dots,t_n)$ and set $t_i' = L/t_i$.  Substitute $x_i^{t_i'}$ into $p(x_1^{t_1},\dots,x_n^{t_n})$ to get $p((x_1^{t_1'})^{t_1},\dots,(x_n^{t_n'})^{t_n}) = p(x_1^L,\dots,x_n^L) = f(x_1^{t_1'},\dots,x_n^{t_n'}) g(x_1^{t_1'},\dots,x_n^{t_n'})$.  Since $f$ and $g$ are not constant we have found a nontrivial factorization of $p(x_1^L,\dots,x_n^L)$\end{proof}

One important process in this paper is {\em homogenization}.  Begin with a polynomial $p(x_1,\dots,x_n)$ and make the substitution $x_i = z_i/z_0$.  Then we multiply by $z_0^{deg(p)}$ to get $P(z_0,\dots,z_n) = z_0^{deg(p)} p(z_1/z_0,\dots,z_n/z_0)$.  We refer to $P$ as the {\em homogeneous counterpart} of $p$.

\begin{lem}\label{lem:hmgFactorIntoHmg}
If a homogeneous polynomial $p(x_0,\dots,x_n)$ factors into two polynomials $f(x_0,\dots,x_n)$, $g(x_0,\dots,x_n)$, then both $f$ and $g$ are homogeneous polynomials.
\end{lem}


\begin{lem}\label{lem:irredIFFHmgIrred}
A polynomial $p(x_1,\dots,x_n)$ is irreducible if and only if the homogeneous counterpart $P(z_0,\dots,z_n)$ is irreducible.
\end{lem}



\begin{lem}\label{lem:strongIrredIFFHomgStrongIrredl}
For a polynomial $p(x_1,\dots,x_n)$ in $R[x_1,\dots,x_n]$ of degree $d$, the following are equivalent.
\begin{itemize}
\item $p(x_1,\dots,x_n)$ is strongly irreducible.
\item $p(x_1^k,\dots,x_n^k)$ is irreducible for all $k > 0$.
\item $P(z_0^k,\dots,z_n^k) = p(z_1^k/z_0^k, \dots, z_n^k/z_0^k)z_0^{kd}$ is irreducible for all $k > 0$
\item $P(z_0,\dots,z_n)$ is strongly irreducible.
\end{itemize}

\end{lem}

\begin{proof} The first equivalence follows from Proposition \ref{prop:irreducibleIFFForAllK}.  The second equivalence follows from Lemma \ref{lem:irredIFFHmgIrred}.  The third equivalence follows from Proposition \ref{prop:irreducibleIFFForAllK}\end{proof}

\begin{lem}\label{lem:strongIrredThenStrongCoprimeToLessVar}
Suppose that $p(x_0,\dots,x_n) \in R[x_0,\dots,x_n]$ is strongly irreducible and the degree of $p$ in $x_i$ is not zero for all $i$.  Then $p$ is strongly coprime to any polynomial in $R[x_0,\dots,\hat{x_i},\dots,x_n]$.
\end{lem}

\begin{proof} To get a contradiction and without loss of generality let $g$ be an element of $R[x_1,\dots,x_n]$ and suppose that $g,p$ are not strongly coprime.  For the free abelian group $A = \Z^{n+1}$ there exists linearly independent sets $\{a_i\}$ for $p$ and $\{b_i\}$ for $g$ such that $p(a_0,\dots,a_n)$ and $g(b_1,\dots,b_n)$ have a common factor.  We can identify $A$ with the commutative multiplicative group generated by $\{x_i\}$ in a way that $b_i = x_i^{t_i}$.  We do not have any control over $a_i$.  More precisely $a_i = x_0^{t_{0,i}}\cdots x_n^{t_{n,i}}$.  Note that the degree of $g$ in $x_0$ is $0$ and since $p$ is strongly irreducible $g(x_1^{s_1},\dots,x_n^{s_n}) = p(a_0,\dots,a_n)f(x_0,\dots,x_n)$.  Also the degree of $g(x_1^{s_1},\dots,x_n^{s_n})$ in $x_0$ is the sum of the degrees of $p(a_0,\dots,a_n)$ and $f(x_0,\dots,x_n)$ in $x_0$.  Observe that one $a_i$ must have positive degree in $x_0$ because if all $a_i$ had degree $0$ in $x_0$ it would follow that there are $n+1$ linearly independent vectors in $\Z^n$ which is impossible.  The fact that at least one $a_i$ has positive degree in $x_0$ is a contradiction to the fact that $g(x_1^{s_1},\dots,x_n^{s_n})$ has degree $0$ in $x_0$ and therefore $g$ and $p$ are strongly coprime\end{proof}

One should view Lemma \ref{lem:strongIrredThenStrongCoprimeToLessVar} as stating that any polynomial which is strongly irreducible is strongly coprime to all polynomials in fewer variables.  

For later results we want to be able to easily compute a sufficient condition for irreducibility.  The following is a sufficient condition for irreducibility.  We work over an algebraically closed field in order to apply some known results of algebraic geometry.

\begin{defn}\label{defn:smooth}
Let $P \in R[x_0,\dots,x_n]$ be a homogeneous polynomial, where $R$ is an algebraically closed field.  We say that $P$ is {\em smooth} if the system $$\frac{\partial P(x_0,\dots,x_n)}{\partial x_i} = 0\mbox{, }i = 0,\dots,n$$ has only the trivial solution.
\end{defn}

Proposition \ref{prop:smooth} is a standard result in algebraic geometry.  See \citep{SHA13} Chapter 2 for a discussion on the subject.

\begin{prop}\label{prop:smooth}
Let $P$ be a homogeneous polynomial in $n+1$ variables, $n \geq 2$, over an algebraically closed field $R$.  If $P$ is smooth then $P$ is irreducible over $R$ and hence over any subring containing all the coefficients of $P$.
\end{prop}

\begin{proof} We prove the contrapositive.  Assume that $P$ is reducible.  By Lemma ~\ref{lem:hmgFactorIntoHmg}, $P = F G$ for some homogeneous polynomials $F$ and $G$  which are not units.  Examining the partial derivatives, we see that $$\frac{\partial P}{\partial x_i} = \frac{\partial F}{\partial x_i}   G + \frac{\partial G}{\partial x_i} F.$$   Next we consider the algebraic sets $Z(F),Z(G) \subset \P^n$, where $Z(F)$ denotes the zero locus of the polynomial.  These are hypersurfaces in $\P^n$, of dimension $n-1.$  Since $n \geq 2$ we have that $n-1 \geq 1$, therefore these hypersurfaces must intersect by the projective dimension theorem ~\citep{hartshorneAlgGeometry}.  Since $Z(F) \cap Z(G) \neq \emptyset$ there exists a point $q = [y_0,\dots,y_n]$ where $y_i \neq 0 $ for some $i$ such that $F(q) = G(q) = 0$.  Thus $q$ is a non trivial solution to $$\frac{\partial P(x_0,\dots,x_n)}{\partial x_i} = 0$$ as desired\end{proof}

Note that if we combine Propositions ~\ref{prop:irreducibleIFFForAllK} and ~\ref{prop:smooth} we get a criterion for strong irreducibility.  Namely, if $p(x_1^k,\dots,x_n^k)$ is smooth for all $k > 0$ then $p$ is strongly irreducible.  But we can get a simpler condition that captures this criterion, as follows.

Since we have a characterization of what being strongly irreducible we can combine Proposition \ref{prop:smooth} and Lemma \ref{lem:strongIrredIFFHomgStrongIrredl} to get a condition.  This condition is summarized in Proposition \ref{prop:criterion}.

\begin{prop}\label{prop:criterion}
Let $P$ be a homogeneous polynomial over an algebraically closed field.  Then $P(x_0^k,\dots,x_n^k)$ is smooth for all $k > 0$ if and only if the system $$x_i \frac{\partial P}{\partial x_i} = 0$$  has only the trivial solution.
\end{prop}

\begin{proof} Assuming $P(x_0^k,\dots,x_n^k)$ is smooth for all $k > 0$.  Then the equation $$\frac{\partial (P(x_0^k,\dots,x_n^k)) }{\partial x_i} = 0$$  has only the trivial solution.  By the chain rule $$\frac{\partial (P(x_0^k,\dots,x_n^k)) }{\partial x_i} = k x_i^{k-1} \frac{\partial P }{\partial x_i}\mid_{ (x_0^k,\dots,x_n^k)}.$$  On inspection we see that for $k  > 1 $ the following two systems have the same solution sets, $$x_i^{k-1} \frac{\partial P }{\partial x_i}\mid_{ (x_0^k,\dots,x_n^k)} = 0$$ and $$x_i^{k} \frac{\partial P }{\partial x_i} \mid_{ (x_0^k,\dots,x_n^k)} = 0.$$  Thus the latter system  has only the trivial solution.  Substituting $y_i = x_i^k$  we see $$y_i   \frac{\partial P}{\partial x_i}\mid_{(y_0,\dots,y_n)} = 0$$  has only the trivial solution.  The proof of this direction is completed by observing that $$\frac{\partial P}{\partial x_i}\mid_{(y_0,\dots,y_n)} = \frac{\partial P(y_0,\dots,y_n)}{\partial y_i}.$$

To prove the other direction we show the contrapositive.  Assume that $P(x_0^k,\dots,x_n^k)$ is not smooth for some $k > 0$.  Therefore  the system $$x_i^{k-1} \frac{\partial P }{\partial x_i} \mid_{ (x_0^k,\dots,x_n^k)} = 0$$ has a nontrivial solution.  Multiplying each equation by the appropriate $x_i$ we get that the system  $$x_i^{k} \frac{\partial P }{\partial x_i}\mid_{ (x_0^k,\dots,x_n^k)}= 0$$ also has a nontrivial solution.  Substituting $y_i = x_i^k $ we get that the system $$y_i   \frac{\partial P}{\partial x_i}\mid_{(y_0,\dots,y_n)}
= 0,$$  has a nontrivial solution.  Observing that $$\frac{\partial P}{\partial x_i}\mid_{(y_0,\dots,y_n)} = \frac{\partial P}{\partial y_i}$$ completes the proof\end{proof}

\begin{cor}\label{cor:ifCriterionThenStronglyIrreducible}
Let $P$ be a homogeneous polynomial in at least 3 variables over an algebraically closed field.  If the system $$x_i \frac{\partial P}{\partial x_i} = 0$$  has only the trivial solution then $P$ is strongly irreducible.
\end{cor}

\begin{proof} By Proposition ~\ref{prop:criterion}, $P(x_1^k,\dots,x_n^k)$ is smooth for all $k > 0$.  Then by Proposition ~\ref{prop:smooth}, $P(x_1^k,\dots,x_n^k)$ is irreducible for all $k > 0$.  Applying Proposition ~\ref{prop:irreducibleIFFForAllK} we get that $P$ is strongly irreducible\end{proof}

Corollary ~\ref{cor:ifCriterionThenStronglyIrreducible} is our most useful tool.  The way to apply it to a nonhomogeneous polynomial $p$ over $\Z$ or $\Q$ is to consider it as a polynomial over $\Comp$ and then homogenize $p$.  If the homogeneous counterpart $P$ is strongly irreducible, then $p$ is strongly irreducible by Lemma ~\ref{lem:irredIFFHmgIrred}.

More formally, we have the following lemma

\begin{lem}\label{lem:completeToRing2}
Let $R$ be an integral domain, let $F$ be the completion of its field of fractions, and let $p$ be in $R[x_0,\dots,x_n]$.  If $p$ is strongly irreducible over $F$ then $p$ is strongly irreducible over $R$.
\end{lem}

\begin{proof} Suppose that $p$ is strongly irreducible over $F$.  Then by Proposition \ref{lem:strongIrredIFFHomgStrongIrredl}, for all $k > 0$, $p(x_1^k,\dots,x_n^k)$ is irreducible over $F$.  Therefore $p(x_1^k,\dots,x_n^k)$ is irreducible over $R$ because $R$ is a subring of $F$.  Thus by Proposition \ref{lem:strongIrredIFFHomgStrongIrredl}, $p$ is strongly irreducible over $R$.\end{proof}

\begin{lem}\label{lem:FinalCriterion}
Suppose that $p \in \Z[x_1,\dots,x_n]$, where $n \geq 2$.  Let $P$ denote the homogenization of $p$.  Suppose that the system $$x_i \frac{\partial P}{\partial x_i} = 0$$ only has a trivial solution over $\mathbb{C}$.  Then $p$ is strongly irreducible over $\Z$ and therefore strongly coprime to all polynomials in fewer variables.  
\end{lem}

\begin{proof} By Corollary \ref{cor:ifCriterionThenStronglyIrreducible}, $P$ is strongly irreducible over $\Comp$.  By Lemma \ref{lem:strongIrredIFFHomgStrongIrredl}, $p$ is strongly irreducible over $\mathbb{C}$.  By Lemma \ref{lem:completeToRing2}, $p$ is strongly irreducible over $\Z$.  By Lemma \ref{lem:strongIrredThenStrongCoprimeToLessVar}, $p$ is strongly coprime to all polynomials in fewer variables\end{proof}

\subsection{Generic Condition}
The property of being strongly irreducible as an element of $\Comp[x_1,\dots,x_n]$ is a generic condition under the Zariski topology.  We will explain this through an example for polynomials in 2 variables.  First we need a couple of preliminaries.  The first preliminary is that homogeneous polynomials can be used to define a locus of zeroes on projective space.  The second is the identification between the set of homogeneous polynomials of fixed degree and a complex affine space.  We give a quick example of how this is done.  This example is well known to algebraic geometers.  Let $Z(p) \subset \Comp^n$ denote the zero set of a polynomial, where $p \in \Comp[x_1,\dots,x_n]$.

\begin{prop}\label{ex:polyToProj}
The set of homogeneous polynomials in three variables of degree 2 over $\Comp$ can be identified with $\Comp^6$, if we include the 0 polynomial or with $\Comp \mathbb{P}^5$ if we identify polynomials using the equivalence relation $p \sim q$ if $Z(p) = Z(q)$.
\end{prop}

\begin{proof} For first claim, observe that all homogeneous polynomials in $\Comp[x_1,x_2,x_3]$ are of the form $a_{1,1}x_1^2 + a_{1,2}x_1x_2 + a_{1,3}x_1x_3 + a_{2,2}x_2^2 + a_{2,3}x_2x_3 + a_{3,3}x_3^2$.  We can identify each polynomial with the point $(a_{1,1}, a_{1,2}, a_{1,3}, a_{2,2}, a_{2,3}, a_{3,3})$.  Then we get each point in $\Comp^6$ except for $(0,0,0,0,0,0)$ which is identified with the $0$ polynomial.

For the second claim we do not include the $0$ polynomial and observe that scaling does not change the zero-locus\end{proof}

In general we can order the monomials of homogeneous polynomials and then take the coefficients of the monomials to define a point.  For homogeneous polynomials of degree $d$ in $n+1$ variables there are $N = {d+n\choose d}$ monomials.  Therefore we can identify a set of homogeneous polynomials of degree $d$ with a subset of $\Comp ^{N}$.  If we consider them equivalent up to scaling (that is they have the same zero sets) then we can identify them with $\Comp\mathbb{P}^{N-1}$.  

For a given degree $d$ being strongly irreducible is a generic condition.  By a generic condition we mean that the condition defines a nonempty open Zariski subset of $\Comp^{N}$.

\begin{cor}\label{cor:genericIrred}
A generic polynomial $p \in \Comp[x_1,\dots,x_n]$ of degree $d$ is strongly irreducible.
\end{cor}

The proof is similar to the proof of generic smoothness for polynomials and is therefore left to the reader.


The following is a proposition that gives us some algebraic information about the group ring $\Z[x_1^{\pm1},\dots,x_n^{\pm1}]$ when localized at a specific multiplication set.  We will use this fact in later sections.

\begin{prop}\label{prop:localizedPID}
Let $\Lambda = \Z[x_1^{\pm1},\dots,x_n^{\pm1}]$ and let $p,q \in \Lambda$ such that $p,q$ are irreducible and relatively prime.  Let $S = \{r \in \Lambda \mid (r,pq) = 1\}$.  Then the ring $S^{-1}\Lambda$ is a PID.
\end{prop}

\begin{proof} First observe that $\Lambda$ is a Noetherian ring, which implies that $S^{-1}\Lambda$ is Noetherian because the ideals of $S^{-1}\Lambda$ are generated by the ideals of $\Lambda$ up to units.  Let $I$ be an ideal in $S^{-1}\Lambda$, then $I = \langle f_1,\dots,f_n \rangle$, where $f_i = g_i^{-1}h_i$ for some $g_i \in S$.  Also, $g_i$ does not share any factors of $p,q$ because $p,q$ are irreducible, $(g_i,pq) = 1$ and $(p,q) = 1$.  Therefore there exist $s_i,t_i$ such that $s_i,t_i$ are maximal and we can rewrite $f_i = p^{s_i}q^{t_i}g_i^{-1}h'_i$.  Since $s_i,t_i$ are maximal it follows that $(h'_i,pq) = 1$ and therefore $h'_i$ is a unit.  Also since $g_i$ is a unit in $S^{-1}\Lambda$ it follows that up to units $I = \langle p^{s_1}q^{t_1},\dots,p^{s_n}q^{t_n} \rangle$.  Note that if $s_i = t_i = 0 $ for some $i$ then $I = S^{-1} \Lambda$.  So we may take $s_i > 0$ or $t_i > 0$.

We next show that if there are at least two generators of $I$ we can reduce the set of generators by one.  Consider $p^{s_1}q^{t_1},p^{s_2}q^{t_2}$, up to symmetry of $s_i,t_i$ we have two cases. The first case is when $s_1 \leq s_2$ and $ t_1 \leq t_2$ in which case $p^{s_2}q^{t_2} = p^{s_1}q^{t_1}p^{s_2-s_1}q^{t_2-t_1}$.  Therefore we can reduce the number of generators by one since $I$ is an ideal.  

The second case is that $s_1 \leq s_2$ and $t_2 \leq t_1$.  Then consider $p^{s_1}q^{t_1} + p^{s_2}q^{t_2} = p^{s_1}q^{t_2}(q^{t_1-t_2}+p^{s_1-s_2})$.  We focus on $q^{t_1-t_2}+p^{s_1-s_2}$  and observe that $(q^{t_1-t_2}+p^{s_1-s_2},pq) = 1$.  To see this suppose for a contradiction $(q^{t_1-t_2}+p^{s_1-s_2},pq) = f$.  Since $f$ divides $pq$, we have three possibilities for $f$.  By unique factorization $f$ is either $p,q$ or $pq$.  By symmetry we may assume $p$ divides $f$; therefore $p$ divides $q^{t_1-t_2}+p^{s_1-s_2}$ which implies that $p$ divides $q$ which contradicts $(p,q) = 1$.  Since $(q^{t_1-t_2}+p^{s_1-s_2},pq) = 1$ it follows that $q^{t_1-t_2}+p^{s_1-s_2} \in S$ and therefore it is a unit in $S^{-1} \Lambda$ so up to units $p^{s_1}q^{t_2} \in I$ and $p^{s_1}q^{t_1} = p^{s_1}q^{t_2}q^{t_1-t_2}$ and $p^{s_2}q^{t_2} = p^{s_1}q^{t_2}p^{s_2-t_1}$.  Thus the ideal $I$ is generated by the set $\{p^{s_1}q^{t_2}, p^{s_3}q^{t_3}, \dots, p^{s_n}q^{t_n} \}$.  By induction on the number of generators we see that $I$ is principally generated\end{proof}

Proposition \ref{prop:localizedPID} is actually true for any finite set of irreducible $\{p_i\}$ with the property that $(p_i,p_j) = 1$ for any $i,j$ and the proof is exactly the same up to permutations for each case.  It also follows from some algebra facts about Dedekind domains.  One can show that $S^{-1} \Lambda$ is a Dedekind domain with finitely many prime ideals and is therefore a PID.

\section{Links With Good Alexander Polynomials}\label{sec:linkPoly}

In this section we construct an examples of links with specified torsion Alexander modules.  We must construct a slice link for later constructions.  This forces the torsion Alexander polynomial to factor but it factors into the form $p(x_1,\dots,x_n)p(x_1^{-1},\dots,x_n^{-1})$.  These polynomials are not irreducible but having $p(x_1,\dots,x_n)$ be strongly irreducible is sufficient for our applications. 

\begin{prop}\label{ex:infFam}
Any member of the following two families is strongly irreducible $$F'_1 = \{p(x_1,\dots,x_{2n}) \mid p(x_1,\dots,x_{2n}) =1-\sum\limits_{i=1}^{2n} (-1)^{i} k_i x_i,  k_i \neq 0 \mbox{ for all }i \}$$ and $$F'_2 = \{q(x_1,\dots,x_{2n+1}) \mid q(x_1,\dots,x_{2n+1}) = 1 + k_2 x_2 + \sum\limits_{i=1}^{2n+1} (-1)^{i} k_i x_i , \forall k_i \neq 0 \mbox{ for all }i \}.$$ We may take $F_i$ a subset of $F'_i$ where the coefficients of polynomials for $F_1$ are subject to the equation $-k_1 + k_2 -k_3 + k_4 - \dots +k_{2n} = 0$, and the coefficients of polynomials in $F_2$ are subject to the equation $-k_1 +2k_2 -k_3 + k_4 -\dots -k_{2n+1} = 0$, and for each $i$, $k_i$ is not equal to $0$.
\end{prop}

The proof is an application of Lemma \ref{lem:FinalCriterion} after homogenizing and thus is left to the reader.

Taking $k_i = 1$ for all $i$ shows that both $F_1$ and $F_2$ from Proposition \ref{ex:infFam} are nonempty.  The equations are there so that if we had a polynomial $p(x_1,\dots,x_n)$ in $F_i$ it follows that $p(1,1,\dots,1) = 1$, which is a condition that we will use to construct ribbon links.  It is easy to see that there are actually infinitely many such polynomials.

\begin{prop}\label{prop:ExistenceOfLink}
For each polynomial $p(x_1,\dots,x_n)$ from Proposition \ref{ex:infFam}, there exists a ribbon link $L$ such that $\Delta_L(x_1,\dots,x_n) = p(x_1,\dots,x_n)p(x_1^{-1},\dots,x_n^{-1})$ and the torsion Alexander module $T\cA$ of $L$ is $\Z[x_1^{\pm1},\dots,x_n^{\pm1}]/\Delta_L$.
\end{prop}

\begin{proof} 

\begin{figure}[h!]
\centering
\begin{subfigure}[b]{0.3\textwidth}
\includegraphics[scale=.25]{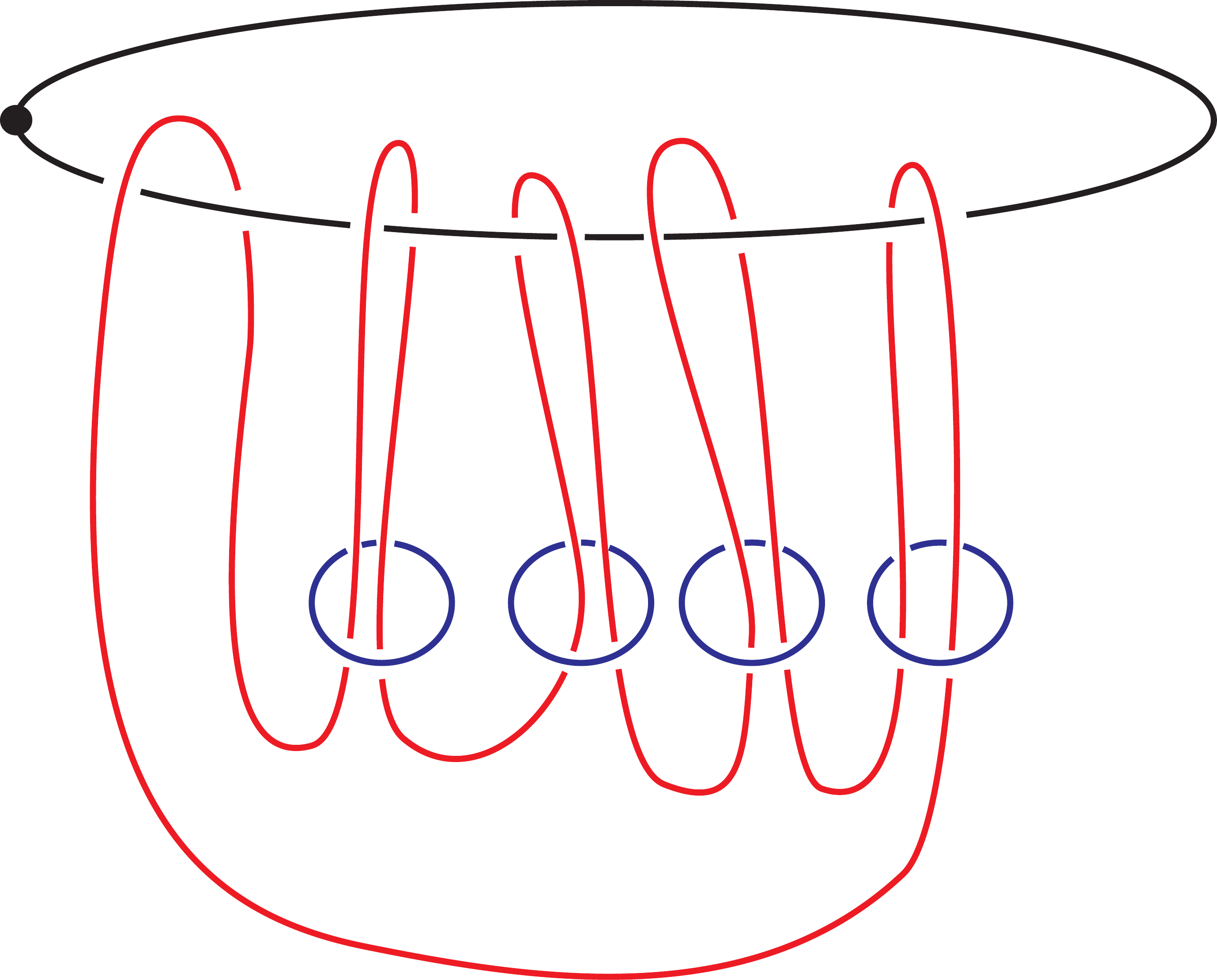}
\put(-220,150){$h_0$}
\put(-157,50){$L_1$}
\put(-125,50){$L_2$}
\put(-95,50){$L_3$}
\put(-65,50){$L_4$}
\put(-190,20){$\alpha$}
\caption{Type 1}
\label{fig:LinkEven}
\end{subfigure} \quad \quad \quad \quad \quad
\begin{subfigure}[b]{0.3\textwidth}
\includegraphics[scale=.25]{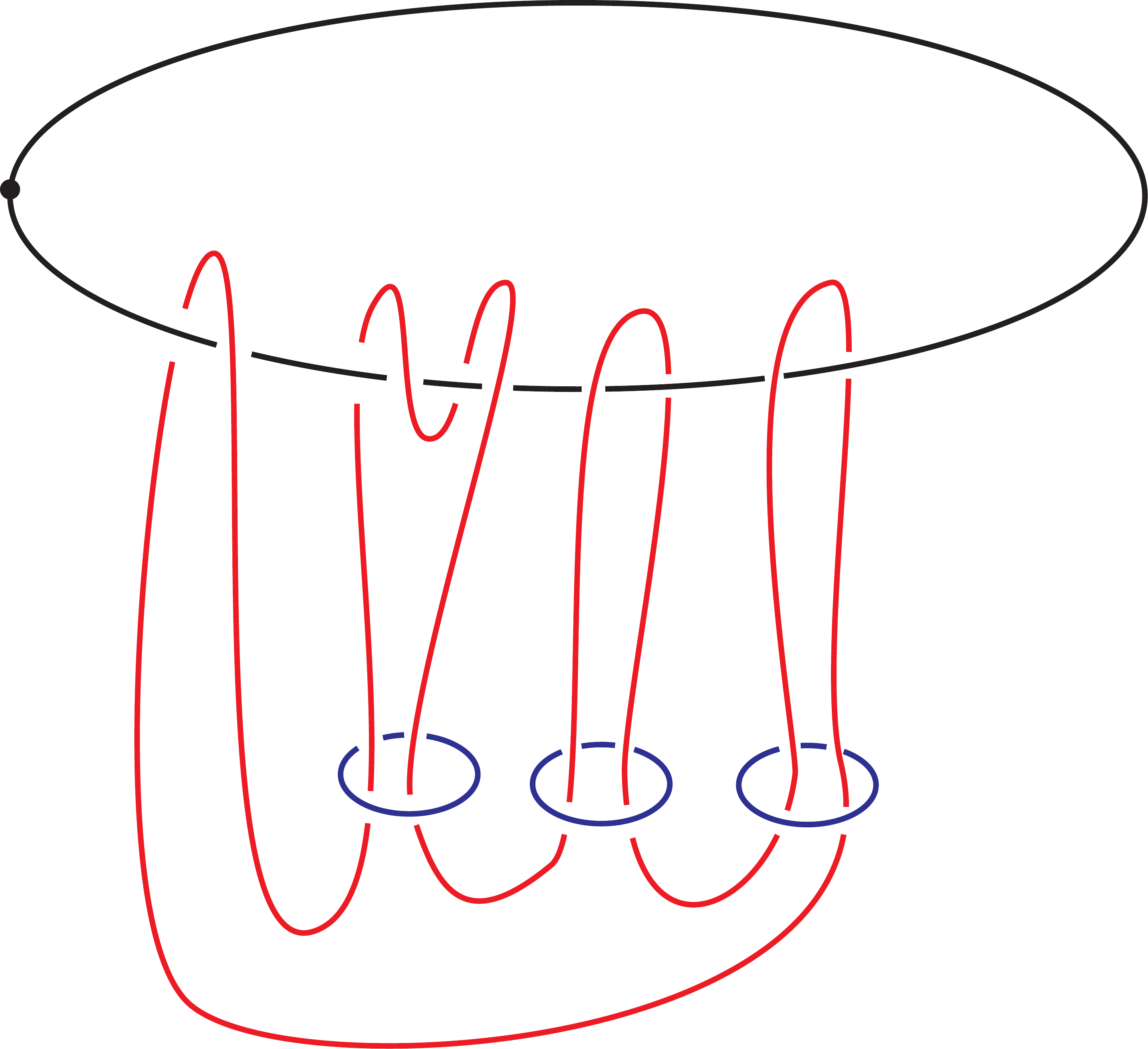}
\put(-200,190){$h_0$}
\put(-159,55){$L_1$}
\put(-123,55){$L_2$}
\put(-85,55){$L_3$}
\put(-200,30){$\alpha$}
\caption{Type 2}
\label{fig:LinkOdd}
\end{subfigure}
\begin{subfigure}[b]{0.3\textwidth}
\includegraphics[scale=.25]{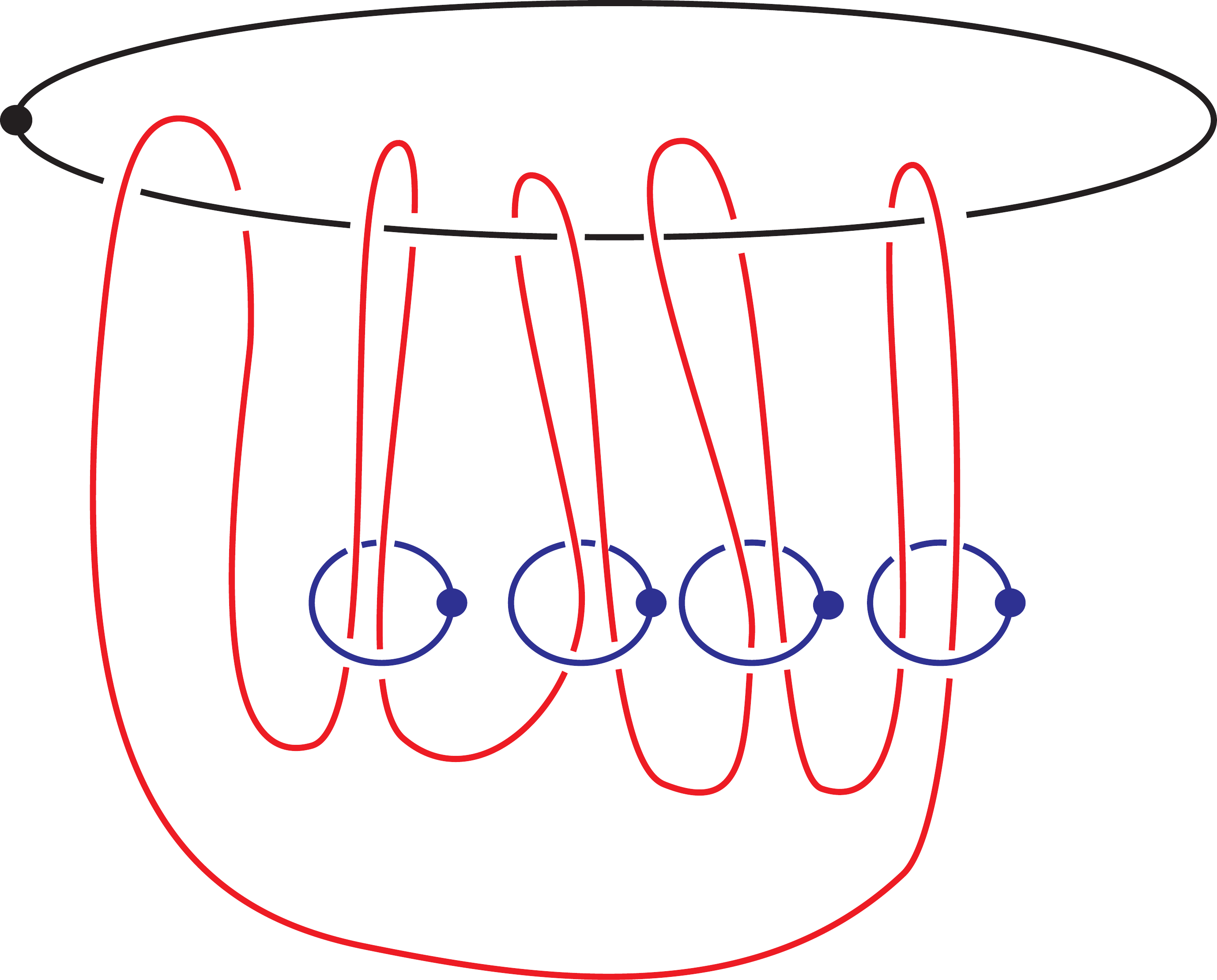}
\put(-220,150){$h_0$}
\put(-157,50){$h_1$}
\put(-125,50){$h_2$}
\put(-95,50){$h_3$}
\put(-65,50){$h_4$}
\put(-190,20){$\alpha$}
\caption{Disk Complement Even}
\label{fig:DiskComplementEven}
\end{subfigure} \quad \quad \quad \quad \quad
\begin{subfigure}[b]{0.3\textwidth}
\includegraphics[scale=.25]{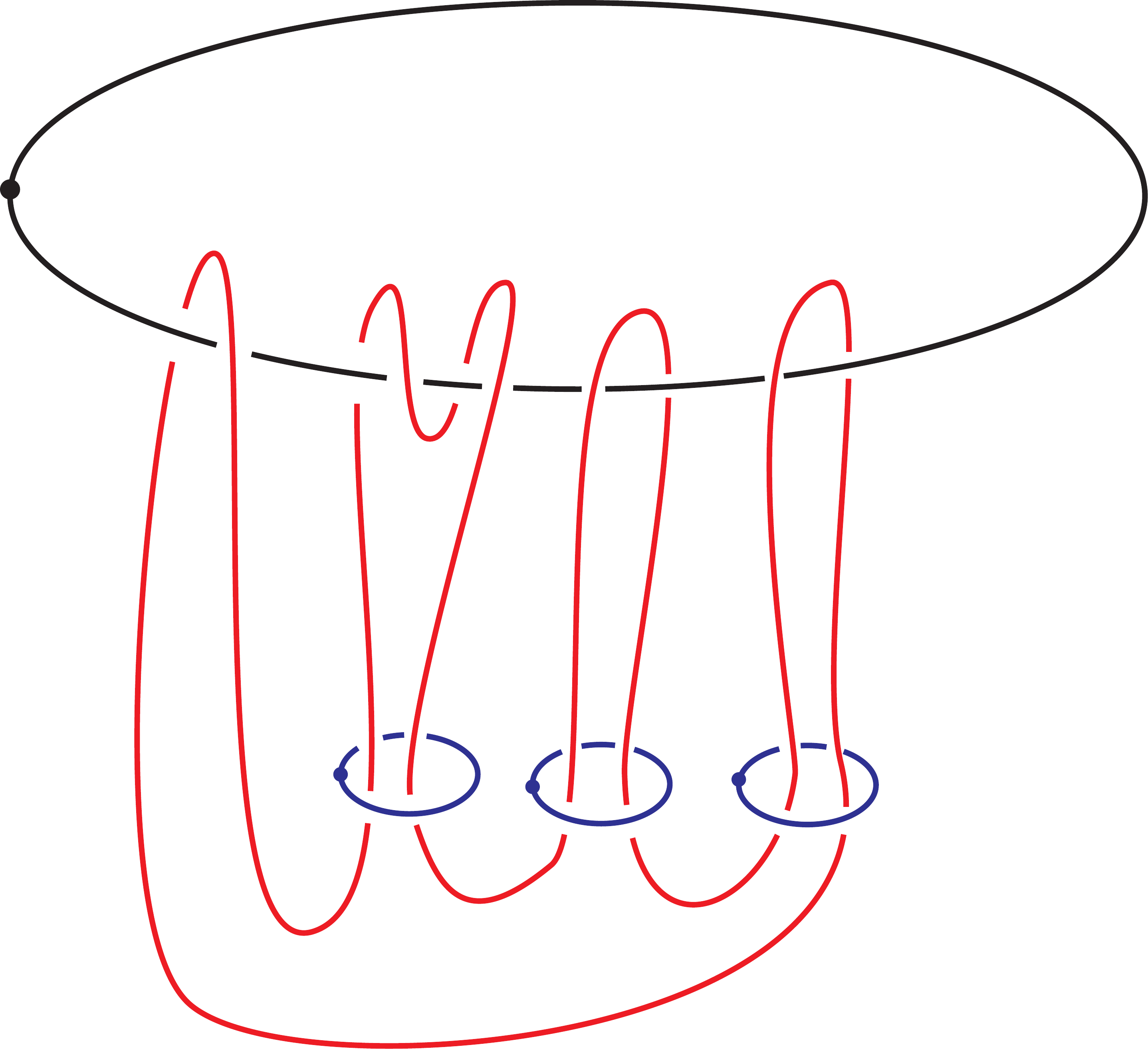}
\put(-220,155){$h_0$}
\put(-155,55){$h_1$}
\put(-122,55){$h_2$}
\put(-85,55){$h_3$}
\put(-190,20){$\alpha$}
\caption{Disk Complement Odd}
\label{fig:DiskComplementOdd}
\end{subfigure} \quad \quad \quad \quad \quad

\caption{}
\label{fig:BigFig}
\end{figure}

Consider Figures \ref{fig:LinkEven} and \ref{fig:LinkOdd}, which generalize to give the desired ribbon link in the following way.  First consider the link given by $L_1,L_2,L_3,L_4$ without the dotted circle or the $\alpha$ curve.  This is the unlink which bounds a set of disjoint disks.  This remains true when we attach the 1-handle $h_0$.  The next step is to attach a 2-handle, call it $a$, with attaching sphere $\alpha$ such that $\alpha,L_1,\dots,L_4$ is an unlink and $\alpha$ goes once over $h_0$.  Since $\alpha$ goes once over $h_0$ geometrically once, it cancels $h_0$ and therefore the resulting 4-manifold is $B^4$ and the image of $L_1,\dots,L_4$ is a slice link in this new $B^4$ (by abuse of notation we refer to these as $L_i$ for the rest of the paragraph).  We will construct links generalizing the ones in Figures \ref{fig:LinkEven} and \ref{fig:LinkOdd}; diagrams like Figure \ref{fig:LinkEven} correspond to polynomials in $F_1$ from Proposition \ref{ex:infFam} and diagrams like Figure \ref{fig:LinkOdd} correspond to polynomials in $F_2$ from Proposition \ref{ex:infFam}.

We construct the desired link in the following way.  The manifold that results from adding $h_0$ and attaching $a$ with framing $0$ is diffeomorphic to $B^4$ because $h_0$ and $\alpha$ are a canceling pair.  Let $f$ denote the diffeomorphism from $B^4 \cup h_0 \cup a$ to $B^4$.  Let $\delta_i$ be the obvious slice disk for $L_i$ in $B^4$.  We see that $f(\delta_i)$ is a slice disk for $f(L_i)$, since $\alpha$ is unlinked from $L_i$, and since $L_i$ is unlinked from $h_0$.  Let $X$ be $B^4 \backslash  (\cup_{i=1}^n f(\delta_i))$ and $\widetilde{X}$ be the universal abelian cover of $X$.  We want to compute the torsion Alexander polynomial for $f(L)$.  This turns out to be the torsion submodule of the Alexander module of $0$-surgery along $f(L)$, call it $M_{f(L)}$.  Since $X$ is the result of removing a product neighborhood the slice disks from $B^4$, it follows that $\partial X = M_{f(L)}$.  To calculate $T\cA(f(L))$ we will look at the universal abelian cover which is actually $\partial \widetilde{X}$.  For the rest of the proof we do not distinguish $L$ and $f(L)$.

We begin with $B^4$ and then attach $n+1$ 1-handles $h_0,\dots,h_n$ and call the result $X_0$.  This is diffeomorphic to having an unlink with $n+1$ components and removing each slice disk $\delta_i$ from $B^4$.  Let $p_0$ be a point in the interior of $B^4$.  Then $\pi_1(X_0) = \langle t,x_1,\dots,x_n \rangle$ where $t$ is the core of $h_0$ along with two arcs from the attaching sphere of $h_0$ to $p_0$ in $B^4$.  For $x_i$ we take arcs from $p_0$ to the attaching spheres of $h_i$ to obtain the other generators.  We then attach a 2-handle $a$ along $\alpha$ with framing $0$, similar to above, passing through the $i$th component then winding $k_i$ times around $h_0$ and then coming out again.  Performing handle slides shows that the link $L$ is a ribbon link.

Look at Figure \ref{fig:LinkOdd}, at one part $\alpha$ goes through $L_1$ and wraps twice around the curve with the dotted circle.  The picture for the general case is similar.  Even though we take the framing of $a$ to be 0, this is not crucial.  Call the resulting manifold $X$.  Then we take the cover $\widetilde{X}$ corresponding to the map $\phi: \pi_1(X) \rightarrow \Z^n$ given by $x_i \mapsto e_i$.  Since $\phi(t) = 0$, by definition, it follows that $h_0$ lifts to a 1-handle in the corresponding cover. The lift is freely permuted by the deck group $\Z^n$.  Notice that the induced cover on $X_0$, $\widetilde{X_0}$, is homotopy equivalent to the universal abelian cover of the wedge of $n$ circles with 1-handles indexed by $\Z^n$ attached equivariantly with respect to the deck group that comes from $\phi$.  The attaching sphere for each 1-handle corresponds to one of the $\Z[\Z^n]$ lifts of the base point $p_0$.  Since $\phi(\alpha) = 0$ by construction, $a$ lifts to the cover, making $\widetilde{X}$ homotopy equivalent to $\widetilde{X_0}$ with $\Z^n$ 2-handles attached to the multiple lifts of $\alpha$.
  
  We see that $ \Z[\Z^n] \oplus H_1(F_n,\Z[\Z^n]) \rightarrow H_1(X,\Z[\Z^n]) $ is a surjection, where $F_n$ is the free group on $n$ letters.  The kernel of the map is the equivariant image of $\alpha$, which is $p(e_1,\dots,e_n)$ by construction.  Therefore $$H_1(X,\Z[\Z^n])  = (\Z[\Z^n]/\langle p(e_1,\dots,e_n) \rangle) \oplus H_1(F_n, \Z[\Z^n]).$$  
  
  Next, we focus on $H_2(X;\Z[\Z^n])$ and observe that the only possible equivariant generator of $H_2(X,\Z[\Z^n])$ would be $a$ since it is the only 2-handle.  Since a lift of $\alpha$ goes over a lift of $h_0$ geometrically once it follows that $\partial \tilde{a} = \tilde{\alpha} \neq 0$.  This implies that $H_2(X,\Z[\Z^n]) = 0$.

We compute $H_2(X,M_L;\Z[\Z^n])$ by computing $H^2(X,\Z[\Z^n])$.  The handle body decomposition of $X$ gives rise to the chain complex $$0 \rightarrow C_2(X,\Z[\Z^n]) \rightarrow C_1(X,\Z[\Z^n]) \rightarrow C_0(X,\Z[\Z^n]) \rightarrow 0.$$  The generator of $C_2(X,\Z[\Z^n])$ as a $\Z[\Z^n]$ module is $ \alpha $ because there is only one equivariant 2-handle.  Next, $C_1(X,\Z[\Z^n]) = \langle t \rangle \oplus \langle x_1,\dots,x_n \rangle$, where $t,x_i$ are generators of $C_1(X,\Z[\Z^n])$ as free $\Z[\Z^n]$ modules obtained by picking a base point and lifting $t,x_i$.  Since $\partial t = 0$, it follows that $t$ is a cycle and $\partial \alpha = t p(e_1,\dots,e_n)$, where $e_i$ generate the deck group $\Z^n$.  

We calculate the cohomology using the chain complex is $0 \leftarrow C^2(X,\Z[\Z^n]) \leftarrow C^1(X,\Z[\Z^n]) \leftarrow C^0(X,\Z[\Z^n])$.  The modules $C^i(X,\Z[\Z^n])$ are finitely generated and free because the modules $C_i(X,\Z[\Z^n])$ are finitely generated and free.  The dual map from $C^1(X,\Z[\Z^n]) \rightarrow C^2(X,\Z[\Z^n])$ is determined by $t^* \rightarrow \alpha p(e_1,\dots,e_n)$ and $x_i^* \rightarrow 0$.  Therefore $H^2(X,\Z[\Z^n]) = \Z[\Z^n]/\langle p(e_1,\dots,e_n) \rangle$.  By Poincar\'{e}, $\overline{H^2(X,\Z[\Z^n])} \cong H_2(X,M_L,\Z[\Z^n])$, therefore $H_2(X,M_L;\Z[\Z^n]) = \Z[\Z^n]/\langle p(e_1^{-1},\dots,e_n^{-1}) \rangle$.
  
We will compute $H_1(M_L; \Z[\Z^n])$ by looking at the long exact sequence of a pair induced by $M_L \rightarrow X$.  Initially the long exact sequence is $$H_{i+1}(X,M_L; \Z[\Z^n]) \rightarrow H_i(M_L ; \Z[\Z^n]) \rightarrow H_i(X; \Z[\Z^n]) \rightarrow H_i(X,M_L; \Z[\Z^n]).$$  Looking at $(X,M_L)$ upside down we get a dual handle body decomposition which only has a $0$-handle, a $2$-handle, $n+1$ $3$-handles, and a $4$-handle.  Since there are no $1$-handles, $H_1(X,M_L; \Z[\Z^n]) = 0$.  We get an exact sequence $$0 \rightarrow H_2(X,M_L;\Z[\Z^n]) \rightarrow H_1(M_L;\Z[\Z^n]) \rightarrow H_1(X;\Z[\Z^n]) \rightarrow 0.$$  
  
    We are only interested in the torsion part of $H_1(M_L; \Z[\Z^n])$ and since $H_1(F_n;\Z[\Z^n])$ is torsion free we get the decomposition $TH_1(M_L;\Z[\Z^n]) \oplus H_1(M_L;\Z[\Z^n])/TH_1(M_L; \Z[\Z^n]) \rightarrow \Z[\Z^n]/\langle p(x_1,\dots,x_n) \rangle \oplus H_1(F_n; \Z[\Z^n]),$ where the map splits across the direct sum.

  Since $H_2(X,M_L;\Z[\Z^n])$ is torsion and injects into $H_1(M_L;\Z[\Z^n])$ it injects into the torsion module so to compute $TH_1(M_L;\Z[\Z^n])$ we have the following exact sequence, $$0 \rightarrow \Z[\Z^n]/\langle p(e_1^{-1},\dots,e_n^{-1})\rangle \overset{\psi}{\rightarrow} TH_1(M_L;\Z[\Z^n]) \overset{\phi}{\rightarrow} \Z[\Z^n]/\langle p(e_1,\dots,e_n) \rangle \rightarrow 0.$$

Finally we compute $TH_1(M_L; \Z[\Z^n])$.  By construction $TH_1(M_L;\Z[\Z^n])$ is cyclically generated by $t$.  We see this by looking at a cover of $S^3 \backslash (L_0 \sqcup \cdots \sqcup L_n)$ where $L_1,\dots,L_n$ is the original link and $L_0$ comes from attaching a 1-handle.  The cover we are looking at is the one associated to the homomorphism $\psi:  \pi_1(S^3 - (L_0 \sqcup L_1 \sqcup \dots \sqcup L_n) ) \rightarrow \Z^n$, where $\psi(\mu_i) = e_i$ for $i$ not equal to $0$ and $\psi(\mu_0) = 0$, where $\mu_i$ is the meridian of $L_i$.  

Let $A_0$ denote the universal abelian cover of $\wedge_{i=1}^n S^1$ and let $p_0$ be a lift of the wedge point.  Then $\Z^n$ acts on $p_0$ by deck translations.  Take $\Z^n$ copies of $I$ parameterize them as $(t,i_1,\dots,i_n)$.  Let $A$ be the space obtained by taking $A_0$ and attaching $\Z^n$ copies of $I$ using a map that identifies $(0,i_1,\dots,i_n)$ and $(1,i_1,\dots,i_n)$ with the point $(i_1,\dots,i_n)p_0$.  The induced cover of $B^4 \backslash (\cup_{i=0}^n \nu(\delta_i))$ is homotopy equivalent $A$.  Each copy of $(i_1,\dots,i_n)p_0 \cup I\times(i_1,\dots,i_n)$ corresponds to a different lift of the meridian of $L_0$ and will be the generator of $TH_1(M_L;\Z[\Z^n])$.  

  So $TH_1(M_L;\Z[\Z^n]) = \Z[\Z^n]/J$ where $J$ is an ideal of $\Z[\Z^n]$.  The ideal $J$ is finitely generated because $\Z[\Z^n]$ is noetherian, therefore $J = \langle f_1,\dots,f_n \rangle$ for some $f_k \in \Z[\Z^n]$.  We analyze the module maps $\phi,\psi$.  Observe that $\phi(f_k) = 0$ for all $k$, therefore $p(e_1,\dots,e_n)$ divides each $f_k$, so each $f_k = p(e_1,\dots,e_n)g_k$ for some $g_k$.  Considering each $f_k$ as an element of $\Z[\Z^n]/\langle p(e_1^{-1},\dots,e_n^{-1}) \rangle$, when we apply $\psi$ we see that $\psi(f_k) = f_k \psi(1) = 0$, so $p(e_1^{-1},\dots,e_n^{-1})$ divides $f_k$.  Since $\Z[\Z^n]$ is a unique factorization domain and since $p(e_1,\dots,e_n)$ and $p(e_1^{-1},\dots,e_n^{-1})$ are coprime it follows that $f_k = p(e_1,\dots,e_n)p(e_1^{-1},\dots,e_n^{-1})h_k$.  This shows that $J \subset \langle p(e_1,\dots,e_n)p(e_1^{-1},\dots,e_n^{-1}) \rangle$.
  
We want to show that $J = \langle p(e_1,\dots,e_n)p(e_1^{-1},\dots,e_n^{-1}) \rangle$, so consider $p(e_1,\dots,e_n) \in \Z[\Z^n] / J$.  It follows that $\phi(p(e_1,\dots,e_n)) = p(e_1,\dots,e_n) \phi(1) = 0$ so $p(e_1,\dots,e_n) \in ker(\phi)$.  Since the sequence is exact $p(e_1,\dots,e_n)$ is an element of $image(\psi)$.   By exactness there exists some $r \in \Z[\Z^n]/\langle p(e_1^{-1},\dots,e_n^{-1}) \rangle$ such that $\psi(r)$ equals  $p(e_1,\dots,e_n)$.  We see that  $p(e_1,\dots,e_n)p(e_1^{-1},\dots,e_n^{-1})$ is also $0$ in $\Z[\Z^n]/J$, since $p(e_1^{-1},\dots,e_n^{-1})r$ is $0$ in $\Z[\Z^n]/\langle p(e_1^{-1},\dots,e_n^{-1}) \rangle$.   This shows that $\langle p(e_1,\dots,e_n)p(e_1^{-1},\dots,e_n^{-1}) \rangle$ is a subset of $J$ which shows that $J$ equals $\langle p(e_1,\dots,e_n)p(e_1^{-1},\dots,e_n^{-1}) \rangle$.  It follows that the torsion Alexander module $T\cA = TH_1(M_L;\Z[\Z^n]) = \Z[\Z^n]/\langle p(e_1,\dots,e_n)p(e_1^{-1},\dots,e_n^{-1}) \rangle$, and that $T\cA$ is cyclic with $\Delta_L(e_1,\dots,e_n) = p(e_1,\dots, e_n)p(e_1^{-1},\dots,e_n^{-1})$.   \end{proof}

\begin{cor}\label{cor:NonTrivBlf}
For the links $L$ constructed in Proposition \ref{prop:ExistenceOfLink} the Blanchfield form is nontrivial, and there exists an $\eta \in A(L)$ such that $\blf(\eta,\eta) \neq 0$.
\end{cor}

\begin{proof} First we state some facts about the types of links in Proposition \ref{prop:ExistenceOfLink}.  For these links $L$ with $n$-components, the coefficient ring we use to twist homology is $S^{-1}\Lambda_{L} = S^{-1} \Z[\Z^n]$, where $S$ is the multiplicative set generated by all the polynomials strongly coprime to $\Delta_L$.  Since $L$ is understood we suppress it from the notation.  For the links from Proposition \ref{prop:ExistenceOfLink} the following are true \begin{itemize}
\item $TH_1(M_L;\Z[\Z^n]) = \Z[\Z^n]/\langle p(x_1,\dots,x_n) p(x_1^{-1},\dots,x_n^{-1}) \rangle$
\item $TH_1(M_L;\Z[\Z^n]) \hookrightarrow TH_1(M_L;S^{-1}\Lambda)$
\item $Hom(TH_1(M_L;S^{-1}\Lambda),K\Gamma/S^{-1}\Lambda) \neq 0$, where $K\Gamma$ is the field of fractions of $\Z[\Z^n]$.
\end{itemize}

For the rest of the proof for a given module left module $M$ over the group ring $\Z \Gamma$ when we write $\overline{M}$ to mean take the induced right module that comes from the group homomorphism $inv: \Z \Gamma \rightarrow \Z \Gamma$ such that $inv(\gamma) = \gamma^{-1}$.  To prove the proposition we first calculate the localized Blanchfield form $\blf_{S^{-1}\Lambda}$.  The localized Blanchfield form is defined in \citep[Theorem 2.3]{CL12} and comes from the commutative diagram in Figure \ref{fig:BFF}.  Observe that by Proposition \ref{prop:localizedPID} $S^{-1}\Lambda$ is a principal ideal domain.

\begin{figure}

\begin{tikzpicture}[xscale = 5,yscale = 2]
\node (A) at (-0.75,3) {$H_2(M_L;K) $};
\node (B) at (-0.75,2) {$\overline{H^1(M_L;K)}$};
\node (C) at (-0.75,1) {$\overline{Hom_{S^{-1}\Lambda}(H_1(M_L;S^{-1}\Lambda);K)}$};
\node (D) at (-0.75,0) {$\overline{Hom_{S^{-1}\Lambda}(TH_1(M_L;S^{-1}\Lambda),K)}$};
\node (E) at (0.75,3) {$H_2(M_L;K/S^{-1}\Lambda)$};
\node (F) at (0.75,2) {$\overline{H^1(M_L;K/S^{-1}\Lambda)}$};
\node (G) at (0.75,1) {$\overline{Hom_{S^{-1}\Lambda}(H_1(M_L;S^{-1}\Lambda);K/S^{-1}\Lambda)}$};
\node (H) at (0.75,0) {$\overline{Hom_{S^{-1}\Lambda}(TH_1(M_L;S^{-1}\Lambda),K/S^{-1}\Lambda)}$};
\path[->,font=\scriptsize,>=angle 90] 
(A) edge node[above]{$\psi$} (E)
(A) edge node[right]{$P.D.$} (B)
(B) edge node[above]{} (F)
(B) edge node[right]{$\kappa$} (C)
(C) edge node[above]{} (G)
(C) edge node[right]{$j$} (D)
(D) edge node[above]{} (H)
(E) edge node[right]{$P.D.$} (F)
(F) edge node[right]{$\kappa$} (G)
(G) edge node[right]{$j$} (H);
\end{tikzpicture}
\caption{Generalized Blanchfield Form}
\label{fig:BFF}
\end{figure}

The Blanchfield form is a map $\blf_{S^{-1}\Lambda}: TH_1(M_L; S^{-1}\Lambda) \times TH_1(M_L; S^{-1}\Lambda) \rightarrow K/S^{-1}\Lambda$, where $K$ is the field of fractions of $S^{-1}\Lambda$.  We use the Bockstein sequence that arises from the short exact sequence $0 \rightarrow S^{-1}\Lambda \rightarrow K \rightarrow K/S^{-1}\Lambda \rightarrow 0$.  This induces a long exact sequence $H_p(M_L;S^{-1}\Lambda) \rightarrow H_p(M_L;K) \rightarrow H_p(M_L;K/S^{-1}\Lambda) \rightarrow H_{p-1}(M_L;S^{-1}\Lambda)$.  Leidy shows in \citep[Theorem 2.3]{CL12} that to define a Blanchfield form it suffices to define one on $H_2(M_L;K/S^{-1}\Lambda)$ with $ker(\blf_{S^{-1}\Lambda}) = im(H_2(M_L;K))$ from the Bockstein sequence.  

The map is defined using the diagram in Figure \ref{fig:BFF} by going down the right column.  More precisely, using Poincare Duality with twisted coefficients, there exists a map $P.D.$ from $H_2(M_L;K/S^{-1}\Lambda)$ to $\overline{H^1(M_L;K/S^{-1}\Lambda)}$.  Compose with the Kronocker evaluation map $\kappa: \overline{H^1(M_L;KS^{-1}\Lambda)} \rightarrow \overline{Hom(H_1(M_L;S^{-1}\Lambda);K/S^{-1}\Lambda)}$.  Compose with the map induced from inclusion $$j:\overline{Hom(H_1(M_L;S^{-1}\Lambda);K/S^{-1}\Lambda)} \rightarrow \overline{Hom(TH_1(M_L;S^{-1}\Lambda);K/S^{-1}\Lambda)}.$$  The composition of these three maps is the Blanchfield form $\blf_{S^{-1}\Lambda} = j \circ \kappa \circ P.D.$ and is well-defined \citep[Theorem 2.3]{CL12}.

The present question is the nontriviality of $\blf_{S^{-1}\Lambda}$ for the specific $S^{-1}\Lambda$ we are using.  We only need to show the composition is not trivial, but this follows because all three maps are surjective.  Poincare Duality is an isomorphism and therefore surjects onto its image.  The Kronecker map $\kappa$ is surjective, this comes from the universal coefficient theorem over the PID $S^{-1}\Lambda$.  The third map $j$ is also surjective because $K/S^{-1}\Lambda$ is a divisible module.  This follows because $Hom(-;K/S^{-1}\Lambda)$ is an exact functor when the target is a divisible module.  Divisibility of $K/S^{-1}\Lambda$ follows from checking the Baer criterion \citep[Thm 6.89 pg. 462]{JR10} and since $S^{-1}\Lambda$ is a PID.  The Blanchfield form is the composition of three surjective maps and is therefore surjective.  The Blanchfield form is not trivial because it surjects onto $Hom(TH_1(M_L;S^{-1}\Lambda),K/S^{-1}\Lambda)$ which is not trivial by construction.

We have shown that the localized Blanchfield form is nontrivial.  There is a relationship between the classical Blanchfield form and the localized Blanchfield form, since $S^{-1}\Lambda$ is a flat $\Z[x_1^{\pm1},\dots,x_n^{\pm1}]$ module.  The relationship is shown in the diagram in Figure \ref{fig:cdBFF}. 
 
 \begin{figure}[h!]
 
 \begin{tikzpicture}[xscale = 7,yscale = 2]
 \node (A) at (0,1) {$H_1(M_L; \Z[x_1^{\pm1},\dots,x_n^{\pm1}])\times H_1(M_L; \Z[x_1^{\pm1},\dots,x_n^{\pm1}])$};
 \node (B) at (0,0) {$H_1(M_L;S^{-1}\Lambda) \times H_1(M_L; S^{-1}\Lambda)$};
 \node (C) at (1,1) {$K/\Z[x_1^{\pm1},\dots,x_n^{\pm1}]$};
 \node (D) at (1,0) {$K/S^{-1}\Lambda$};
 \path[->,font=\scriptsize,>=angle 90] 
 (A) edge node[right]{$i^*$} (B)
 (A) edge node[above]{$\blf$} (C)
 (B) edge node[above]{$\blf_{S^{-1}\Lambda}$} (D)
 (C) edge node[right]{$i^*$} (D);
 \end{tikzpicture}
 
 \caption{Blanchfield Forms and Localized Blacnchfield Forms}
 \label{fig:cdBFF}
 \end{figure}

Therefore the classical Blanchfield form is non trivial and since the torsion Alexander module is cyclic and generated by some curve $\eta$ it follows that $\blf(\eta,\eta)\neq 0$\end{proof}

\section{Filtration and Localization}\label{sec:filtration}

In this section we review some parts of \citep{CHL11} and \citep{JB14} for completeness.

\begin{defn}\label{def:ptfa}\citep[Proposition 2.2]{CHL11}
A group $G$ is poly torsion free abelian if it admits a finite subnormal series $\langle 1 \rangle  \triangleleft G_n \triangleleft G_{n-1}  \triangleleft \cdots  \triangleleft G_0 = G$ such that the factors $G_i/G_{i+1}$ are torsion free abelian.
\end{defn}

\begin{defn}\label{def:commutatorSeries}\citep[Definition 2.1]{CHL11}
A {\em commutator series} is a function $*$ that assigns to each group $G$ a nested sequence of normal subgroups $$\dots \triangleleft G_*^{(n+1)} \triangleleft G_*^{(n)} \triangleleft \dots \triangleleft G_*^{(0)} = G,$$ such that $G_*^{(n)}/G_*^{(n+1)}$ is a torsion free abelian group.  A {\em functorial commutator series} is one that is a functor from the category of groups to the category of series of groups, that is, a commutator series such that, for any group homomorphism $f: G \rightarrow \pi$, $f(G_*^{(n)}) \subset \pi_*^{(n)}$ for each $n$.  If $G_*^{(i)}$ is defined only for $i \leq n$, then this will be called a {\em partially defined commutator series}.

\end{defn}

\begin{defn}\label{def:weakFunctorial}\citep[Definition 2.4]{CHL11}
A commutator series $\{G_*^{(n)}\}$ is {\em weakly functorial} if, for any homomorphism $f: G \rightarrow \pi$ that induces an isomorphism between $G/G^{(1)}_r$ and $\pi/\pi^{(1)}_r$, where $\pi^{(1)}_r$ is from the rational derived series (that is $f$ induces an isomorphism on $H_1(-;\Q)$).
\end{defn}

The commutator series was defined in \citep{CHL11} and behaves like the rational derived series.  For the rest of the paper $G^{(n)}$ refers to the $n$-th term in the derived series and $G^{(n+1)}_*$   Here is a small useful lemma about commutator series.

\begin{lem}\citep[Proposition 2.2 (1)]{CHL11}\label{lem:usefulLem}
Fixing $i$, if $g$ is an element of $(G^{(i)}_*)^{(k)}$, then $g \in G^{(i+k)}_*$.
\end{lem}

\begin{proof} We prove this by induction on $k$.  The base case $k = 0$ is true because $G^{(i)}_* = G^{(i)}_*$.  Assume that $g \in (G^{(i)}_*)^{(n+1)}$.  Then $g = \prod [\gamma_i,\eta_i]$ where $\gamma_i,\eta \in  (G^{(i)}_*)^{(n)}$.  By the inductive hypothesis $\gamma_i,\eta \in  G^{(i+n)}_*$ and therefore $g \in (G^{(i+n)}_*)^{(1)}$.  Applying the inductive hypothesis again we obtain that $g \in G^{(i+n+1)}_*$, as desired\end{proof}

Modified $n$-solvable filtrations arise from different commutator series.  These filtrations are defined as follows.

\begin{defn}\label{def:filtration}\citep[Definition 2.3]{CHL11}
A string link $L$ is an element of $\F_{n}^{*}$ if the zero-framed surgery $M_L$ bounds a compact smooth spin 4-manifold $W$ such that 
\begin{enumerate}
\item $H_1(M_L; \Z) \rightarrow H_1(W;\Z)$ is an isomorphism;
\item $H_2(W;\Z)$ has a basis consisting of connected compact oriented surfaces $\{L_i,D_i \mid 1 \leq i \leq r\}$, embedded in $W$ with trivial normal bundles, wherein the surfaces are pairwise disjoint except that, for each $i,$ $L_i$ intersects $D_i$ transversely once with positive sign;
\item  For each $i$, $\pi_1(L_i) \subset \pi_1(W)_*^{(n)}$ and $\pi_1(D_i) \subset \pi_1(W)_*^{(n)}$;

\noindent A knot $K$ is an element of $\F_{n.5}^*$ if in addition 
\item for each $i$, $\pi_1(L_i) \subset \pi_1(W)_*^{(n+1)}$
\end{enumerate}
\end{defn}

If $L$ is in $\F_{n}^{*}$, we say that $L$ is $(n,*)$-solvable, and the manifold $W$ is an $(n,*)$-solution.  If $L$ is in $\F_{n.5}^{*}$ we say that $L$ is $(n.5,*)$-solvable and the manifold $W$ is an $(n.5,*)$-solution.  To prove the main result we need to find a link $L$ that has $\ell$ components with the property that $L$ is $(1,*)$-solvable for a specific commutator series $*$ but not $(2,*)$-solvable.  We also want the additional property that if a link $L'$ has fewer components than $L'$ is forced to be $(2,*)$-solvable.  It is necessary to construct a commutator series that is tailored to $L$.  The construction is based on looking at terms in higher order Alexander modules and localizing the coefficient ring.  For the following definitions $\Gamma$ is a group and we use the group ring $\Q[\Gamma]$.  In practice $\Gamma = \pi_1(W)/\pi_1(W)^{(n)}_*$.  We need the localization tools found in \citep{JB14} which we review.

\begin{defn}\label{defn:localSets}\citep[Definition 4.7]{JB14}
Suppose $A$ is a normal subgroup of $\Gamma$ and suppose that $A$ is a torsion free abelian group and $\Q[ \Gamma]$ is a right Ore domain.  If $p \in \Q[t_1^{\pm1},\dots,t_n^{\pm1}]$ is non-zero then set $$S_p^{\Gamma,A} = \{\prod\limits_{i=1}^{r}q_i(a_{i,1},\dots,a_{i,s_i}) \mid p,q_i \mbox{ are strongly coprime, } q_j(1,\dots,1) \neq 0, a_{i,j} \in A\}.$$
\end{defn}

When $\Gamma$ and $A$ are understood we will suppress them from the notation.  One thing to note for those who have read \citep{JB14} is that our definition of strongly coprime differs from \citep{JB14} but as sets the multiplicative set $S^{\Gamma,A}_p$ are the same.  The difference is that we require $\{a_1,\dots,a_n\}$ and $\{b_1,\dots,b_n\}$ to be linearly independent sets in the definition of strongly coprime.  Burke does not require the set $\{b_1,\dots,b_n\}$ to be linearly independent.  Using the linear dependence in Burke's definition, it is easily checked that the $S_p^{\Gamma,A}$ are the same sets.  

\begin{prop}\label{prop:divSet}(\citep[Corollary 4.3]{CHL11},\citep[Proposition 4.8]{JB14})
$S_p^{\Gamma,A}$ is a right divisor set.
\end{prop}

Since $S_p$ is a right divisor set then it makes sense to consider the module $\Q \Gamma S_p^{-1}$ which we think of as ``$\Q \Gamma$ localized at $p$''.  For right $\Q \Gamma$ modules $M$ we get $MS_p^{-1} = M \otimes_{\Q\Gamma} \Q \Gamma S_p^{-1}$ which we call  ``$M$ localized at $p$''.  Note that $\Q \Gamma S_p^{-1}$ is a $\Q \Gamma - \Q \Gamma S_p^{-1}$ bimodule.  We apply this to groups in the following manner.  Observe that $G/G_*^{(n)}$ acts on $G_*^{(n)} /G_*^{(n+1)}$, the action is that for $\gamma \in G/G_*^{(n)}$ and $g \in G_*^{(n)} /G_*^{(n+1)}$ $\gamma*g = \gamma g \gamma^{-1}$.  We see that $G_*^{(n)} /G_*^{(n+1)}$ is a right $\Z[ G/G_*^{(n)}]$ module.  

Let $P = (p_1(x_1,\dots,x_n),\dots,p_n(x_1,\dots,x_n))$ where $p_j(x_1,\dots,x_n) \in \Q[x_1^{\pm1},\dots,x_n^{\pm1}]$.  

\begin{defn}\label{defn:pSeries}{\citep[Definition 4.10]{JB14} }The {\em derived series localized} at $P$ is given by $G_P^{(0)} = G$ and for $n \geq 0$, $$G_P^{(n+1)} = ker \left( G_P^{(n)} \xrightarrow{\phi_1} \dfrac{G_P^{(n)}}{[G_P^{(n)},G_P^{(n)}]} \xrightarrow{\phi_2} \dfrac{G_P^{(n)}}{[G_P^{(n)},G_P^{(n)}]} \otimes_{\Z[G/G_P^{(n)}]} \Q[G/G_P^{(n)}] S_{p_n}^{-1}\right)$$ where $S_{p_n} = S_{p_n}^{G/G_P^{(n)},G_P^{(n-1)}/G_P^{(n)}}$ and if $n = 0$ $S_{p_n} = \{1\}$

\end{defn}

Looking closely at Definition \ref{defn:pSeries} we see that a derived series localized at $P$ is a commutator series since the map $\phi_1$ is abelianization and the map $\phi_2$ can be viewed as taking a tensor with $\Q$ to kill $\Z$ torsion and then killing all the $S_{p_n}$ torsion.

\section{Torsion Doubling Operators and Main Results}\label{sec:doubOp}

The following construction is important.  We use it through the rest of this paper.  It is a cobordism that we use repeatedly when constructing $n$-solutions and $(n,*)$-solutions.  First recall that for an infection $R_T(L)$ there are three pieces of data, the string links $R,L$ and a special embedding of $T$ into the complement of $R$ in $S^3$.  

\begin{defn}\label{defn:cobordismConstruction}  Let $Z = M_R \times [0,1] \sqcup_f M_L \times [0,1]$ where $f$ identifies $T$, the embedded exterior of the trivial string link in $M_R \times \{0\}$, with the handlebody used to construct the zero surgery of $M_L \times \{0\}$ (see Definition \ref{def:zeroSurgery} and Definition \ref{def:infectRbySL}).  The map $f$ identifies the preferred longitudes of $T$ with the preferred longitudes of $L$ and the meridians of $T$ with the meridians of $L$.  Observe that $Z$ is a cobordism from $M_{R(L)}$ to $M_R \sqcup M_L$.  The cobordism $Z$ is the {\em crucial cobordism}.
\end{defn}

\begin{figure}[h!]

\includegraphics[ width = 8cm]{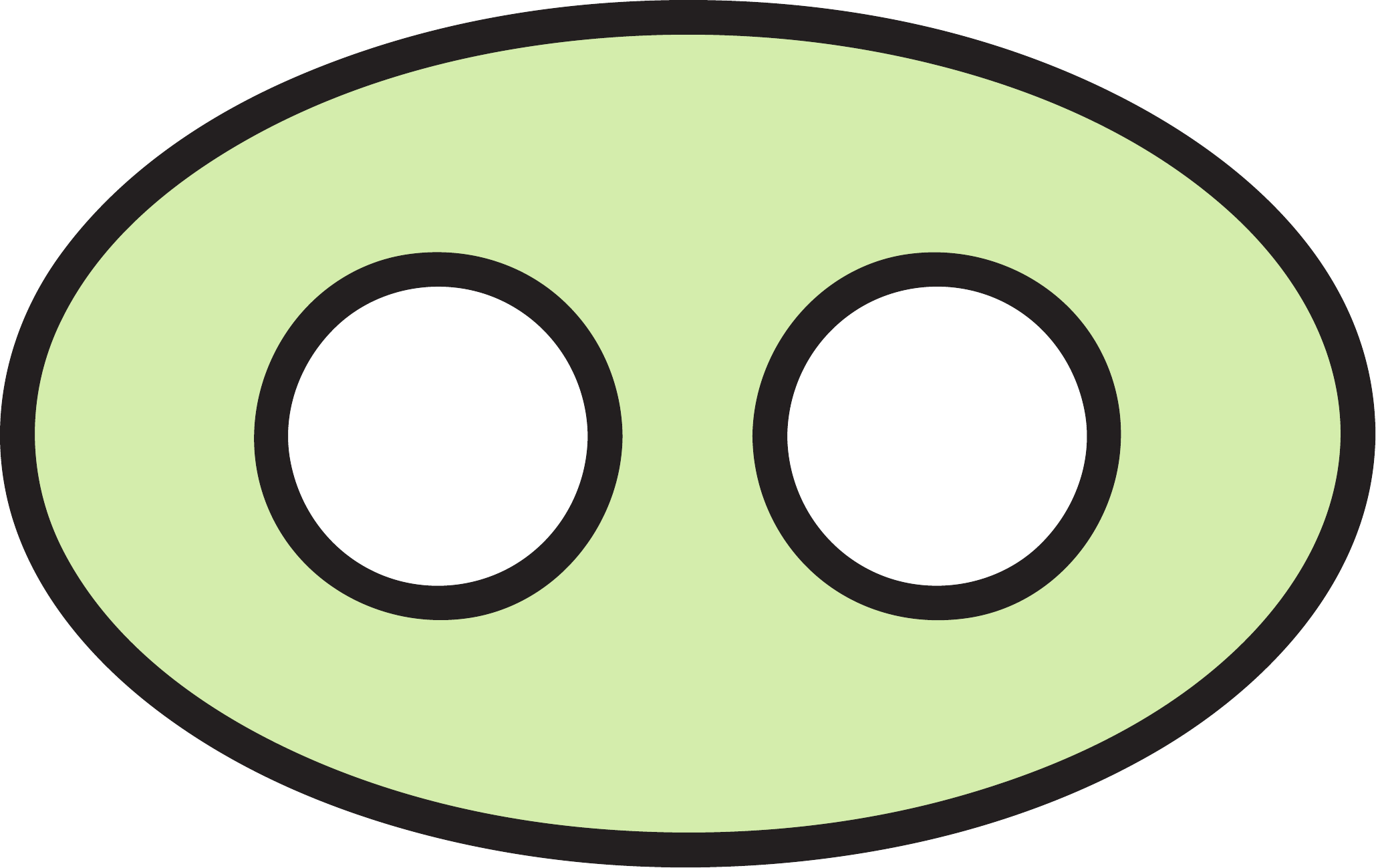}
\put(-170,70){$M_R$}
\put(-70,70){$M_L$}
\put(-200,0){$M_{R_T(L)}$}
\caption{The Crucial Cobordism}
\label{fig:CrucialCobordism}
\end{figure}

We prove a lemma about the fundamental group of the cobordism constructed in Definition \ref{defn:cobordismConstruction}.  

\begin{lem}\label{lem:cobordismPi}
Let $Z$ denote the cobordism constructed in Definition \ref{defn:cobordismConstruction} with $R,L$ as above and assume that $f_*(\pi_1(T)) \subset \pi_1(M_R)^{(1)}$.  Then $\pi_1(Z)$ is normally generated by the image of $\pi_1(M_{R(L)})$ under the map induced by inclusion and therefore normally generated by $\pi_1(R)$.  Also $\pi_1(D^2 \times I \backslash L) \hookrightarrow \pi_1(Z)^{(1)}$.
\end{lem}

\begin{proof} For the first part, apply the Seifert van Kampen Theorem with $U = M_R \times I$ and $V = M_L \times I$.  The second part follows since $f_*(\pi_1(T)) \subset \pi_1(M_R)^{(1)}$ and since the meridians of $T$ normally generate $\pi_1(M_L)$\end{proof}

\begin{lem}{\citep[Lemma 2.1]{JB14}}\label{lem:cobordismH}
Let $Z$ denote the cobordism constructed in Definition \ref{defn:cobordismConstruction} with $R,L$ as above and assume that $f_*(\pi_1(T)) \subset \pi_1(R)^{(1)}$.  Then $H_1(Z; \Z) \cong H_1(M_R; \Z)$ and $H_2(Z;\Z) \cong H_2(M_R;\Z) \oplus H_2(M_L;\Z)$
\end{lem}


For infections the particular choice of embedding of $T$ forces some bounds on the solvability.  These bounds come from the image of $\pi_1(T)$ in $\pi_1(S^3-R)$.  This is made precise in Proposition \ref{prop:Solvable}.  

\begin{prop}\label{prop:Solvable}
Let $*$ be a weakly functorial commutator series.  Let $R$ be an $(n,*)$-solvable string link, let $T$ be as in Definition \ref{def:infectRbySL} and let $L$ be an $m$-solvable string link.  Also assume $m,n \geq 0$.  If $f_*(\pi_1(T)) \subset {\pi_1}(M_R)_*^{(k)}$ with $k > 0$ then $R_T(L)$ is $(min(n,k+m),*)$-solvable.
\end{prop}

\begin{proof} Let $t = min(n,k+m)$, we construct the $(k,*)$-solution as follows.  Let $Z$ be the crucial cobordism from Definition \ref{defn:cobordismConstruction}.  Let $W_R$ denote the $(n,*)$-solution for $M_R$ and $W_L$ denote the $m$-solution $M_L$ and let $C = Z \sqcup W_L \sqcup W_R$ where we attach $W_R$ to $M_R \times \{1\}$ and attach $W_L$ to $M_L \times \{1\}$.  Let $U = W_R \cup M_R \times [0,1]$ and $V = W_L \cup M_L \times [0,1]$.  The reduced Mayer-Vietoris sequence gives the following exact sequence, $$\widetilde H_2(T) \rightarrow \widetilde H_2(U) \oplus \widetilde H_2(V) \rightarrow \widetilde H_2(C) \rightarrow \widetilde H_1(T) \rightarrow \widetilde H_1(U) \oplus \widetilde H_1(V) \rightarrow \widetilde H_1(C) \rightarrow \widetilde H_0(T).$$  Since $T$ is homotopy equivalent to a wedge of circles we see that $\widetilde H_2(T) = \widetilde H_0(T) = 0$ so we have the sequence $$0 \rightarrow \widetilde H_2(U) \oplus \widetilde H_2(V) \rightarrow \widetilde H_2(C) \rightarrow \widetilde H_1(T) \rightarrow \widetilde H_1(U) \oplus \widetilde H_1(V) \rightarrow \widetilde H_1(C) \rightarrow 0.$$  Notice that $(f^{-1})^* : H_1(T) \rightarrow  H_1(M_L)$ is an isomorphism by construction and that $f_*:H_1(T) \rightarrow H_1(M_R)$ is the zero map since $k > 0$.  Therefore $H_1(T) \cong H_1(M_L) \cong H_1(V)$ and $H_1(T) \rightarrow H_1(U)$ is the zero map because $W_R$ and $W_L$ are $(n,*)$-solutions and $m$-solutions respectively.  Then $H_1(C) \cong H_1(M_R)$ from the exact sequence.  We also see that $H_2(C)$ is generated by the classes in $H_2(U)$ and $H_2(V)$ from the exact sequence.

We go over the conditions of solvability.  Under inclusion the meridian that generates $H_1(M_R)$ is isotopic to the meridian that generates $H_1(M_{R_T(L}))$.  Therefore $H_1(M_R)$ is isomorphic to $H_1(M_{R_T(L)})$ under the inclusion.  It also follows that $H_1(C)$ is isomorphic to $H_1(M_{R_T(L)})$.  For the second condition in Definition \ref{def:filtration} take $L_{i,R},D_{i,R}$ to be the surfaces corresponding to the $(n,*)$ solution $W_R$ and $L_{j,L},D_{j,L}$ to be the surfaces corresponding to the $m$-solutions $W_L$.  Since $H_2(C) \cong H_2(U) \oplus H_2(V) \cong H_2(W_R) \oplus H_2(W_L)$ we have that the set $ \{L_{i,R},D_{i,R},L_{j,L},D_{j,L}\}$ forms a basis for $H_2(C)$ satisfying the second condition because the surfaces come from an $(n,*)$-solution and an $m$-solution.  For the third condition in Definition \ref{def:filtration} we examine the fundamental group of the surfaces.

We have shown that $\pi_1(W_R)/\pi_1(W_R)^{(1)} \cong \pi_1(C)/\pi_1(C)^{(1)}$ so by weak functorality of $*$, $\pi_1(W_R)_*^{(k)} \subset \pi_1(C)_*^{(k)}$ for all $k$.  Therefore $\pi_1(L_{i,R}) \subset \pi_1(W_R)_*^{(n)} \subset \pi_1(C)_*^{(n)}$.  The same proof shows that $\pi_1(D_{i,R}) \subset \pi_1(C)_*^{(n)}$.

We claim that $\pi_1(W_L) \subset \pi_1(C)_*^{(k)}$.  If the claim is true then by functorality of the derived series we have $\pi_1(W_L)^{(l)} \subset (\pi_1(C)_*^{(k)})^{(l)}$ and by Proposition \ref{lem:usefulLem} we have $(\pi_1(C)_*^{(k)})^{(l)} \subset \pi_1(C)_*^{(k+l)}$.  We can conclude that $\pi_1(L_{j,L})$ and $\pi_1(D_{j,L})$ are subsets of $\pi_1(W_L)^{(m)}$ which is a subset of  $\pi_1(C)_*^{(k+m)}$.  Therefore each of $\pi_1(L_{i,R})$, $\pi_1(D_{i,R})$, $\pi_1(L_{i,L})$, and $\pi_1(D_{i,L})$ is a subset of $\pi_1(C)_*^{(min(n,k+m))}$.

We need to the following result to prove the claim.  Proving the claim finishes the proof.

\begin{lem}{\citep[Lemma 6.5]{CHL08}}\label{lem:CHL08} Suppose $\phi : A \rightarrow B$ is a group homomorphishm that is surjective on abelianizations.  Then for any positive integer $k$, $\phi(A)$ normally generates $B/B^{(k)}$.\end{lem}  

Observe that $\pi_1(T)$ normally generates $\pi_1(M_L)$ under inclusion so the induced map on abelianizations is surjective. Applying Lemma \ref{lem:CHL08} one sees that $\pi_1(T)$ normally generates $\pi_1(W_L)/\pi_1(W_L)^{(l)}$.  By the hypothesis $f_*(\pi_1(T)) \subset \pi_1(M_R)_*^{(k)}$  By weak functorality of $*$, $f_*(\pi_1(T))  \subset \pi_1(C)_*^{(k)}$.  By abuse of notation $\pi_1(T)$ normally generates $\pi_1(W_L)$ up to elements in $\pi_1(W_L)^{(l)}$ for a fixed $l$ which can be chosen.  Choose $l$ such that $l > n$ and $l > k+m$; then a generating set for $\pi_1(W_L)$ is given by curves $\gamma = \eta \delta \eta^{-1} \beta$ where $\eta \in \pi_1(W_L)$ $\beta \in \pi_1(W_L)^l$ and $\delta \in \pi_1(T) \subset \pi_1(C)_*^{(k)}$.  Since we use a normal series it follows that $\eta \delta \eta^{-1} \in   \pi_1(C)_*^{(k)}$.  By functorality of the derived series $\beta \in \pi_1(C)^{(l)} \subset \pi_1(C)_*^{(l)}$ which implies that $\gamma \in \pi_1(C)_*^{(k)}$.  Since the $\gamma$ curves are generators for $\pi_1(W_L)$ this shows the claim\end{proof}

With respect to Proposition \ref{prop:Solvable}, we would like to strengthen the lemma by having $L$ be $(m,*)$-solvable.  This can only be guaranteed when $(G^{(1)}_*)^{(b)}_*$ is a subset of $G^{(a+b)}_*$, which is not necessarily true.  Additionaly if $R$ is a slice knot then $R_T(L)$ is $(m+k)$-solvable which was shown in \citep{JB14} and \citep{CHL11}.

\begin{defn}\label{defn:torOp}
We say that an infection $R_T$ is a {\em torsion doubling operator} if $R$ is a ribbon link and the image of $H_1(T;\Z)$ under inclusion is a subset of $T\cA(R)$.
\end{defn}

\begin{defn}\label{defn:robustTorOp}
We say that a torsion doubling operator $R_T$ is {\em robust} if for each generator $\eta_i$ of $H_1(T;\Z)$, we have that $\blf(\eta_i,\eta_i) \neq 0$ and for any choice of $c_i \in \Z$, $\blf(\Sigma c_i \eta_i, \Sigma c_i \eta_i) = 0$ implies that $c_i = 0$.
\end{defn}

Next we outline the arguments of \citep{CHL11} and \citep{JB14} to isolate specific infections in different filtrations.

\begin{defn}\citep[Definition 5.1]{JB14}\label{defn:multiCoprime}
Given 
$$P = (p_1(t_1,\dots,t_n),\dots,p_n(t_1,\dots,t_n)), \quad Q  = (q_1(t_1,\dots,t_n),\dots,q_n(t_1,\dots,t_n)),$$ 
we say that $P$ is {\em strongly coprime} to $Q$ if, for some $k \geq 1$ $\widetilde{(p_k,q_k)} = 1$; otherwise we say that $P$ is {\em isogenous} to $Q$

\end{defn}

Observe that given $P = (p_1,\dots,p_n)$ the derived series localized at $P$ (Definition \ref{defn:pSeries}) gives us a commutator series.  When the $P$ is specified when we say $(n,P)$ we are using the associated commutator series.

\begin{defn}\label{defn:vectorCoprime}
We say that $p(t_1,\dots,t_n)$ is {\em strongly coprime} to the vector $$Q = (q_1(t_1,\dots,t_n),\dots,q_n(t_1,\dots,t_n)),$$ if $p$ and $q_i$ are strongly coprime for all $i$.
\end{defn}

Theorem \ref{lem:Triviality} is a theorem that tells us when a knot concordance class is trivial as an element of $\F_{n}^P/\F_{n+1}^P$.  We know that if we perform an iterated doubling operator $n$ times, the result is knot which is $n$-solvable.  What Theorem \ref{lem:Triviality} says is that under this new filtration if the higher order Alexander modules do not support a certain type of torsion then they become $(n+1,*)$-solvable.  The converse is not necessarily true.

\begin{thm}  \label{lem:Triviality}{\citep[Thm 5.2]{JB14}}
Let $L = R_{T_n}^n \circ \cdots \circ R_{T_1}^1(L_0)$ with $L_0 \in \F_0^{m}$ where each $R_{T_j}^j$ is a torsion doubling operator.  Let $Q = (q_n(t_1,\dots,t_n),\dots,q_1(t_1,\dots,t_n))$ where each $q_k(t_1,\dots,t_n) \neq 0$ annihilates $H_1(T_k)$ in $T\cA(R_k)$.  If the vector $P = (p_1(t_1,\dots,t_n),\dots,p_n(t_1,\dots,t_n))$ is strongly coprime to $Q$ then $L \in \F_{n+1}^{P}$
\end{thm}

 \begin{lem}\label{lem:knotFam}
 For any $m > 0$, there exists a ribbon knot $J_m$ such that for each $1 \leq i \leq m$ there is an $\eta_i \in T\cA(J_m)$ such that the set $\{\eta_i\}$ is linearly independent over $\Z$, $\blf(\eta_i,\eta_i) \neq 0$ and $\blf(\eta_i,\eta_j) = 0$ for $i \neq j$.
 \end{lem}
 
 \begin{proof} 
 Let $R_m = \#_{i=1}^m K_i$ where $K_i$ is the ribbon knot shown in Figure \ref{fig:BurkeKnot}.  It was shown in \citep[Example 4.10]{CHL11}  that $\Delta_{K_i}$ are strongly coprime for $i \neq j$.  By additivity of the Blanchfield form $\blf_{R_m} = \Sigma \blf_{K_i}$, and therefore $\blf(\eta_i,\eta_j) = 0$ for $i \neq j$.  One can use the Blanchfield form to show that the set $\{\eta_i\}$ is linearly independent in $T\cA(K_n)$ by localizing at the different $\Delta_{K_i}$ and showing that the localized form vanishes on all but one.

\begin{figure}[h!]
\includegraphics[ width = 7cm]{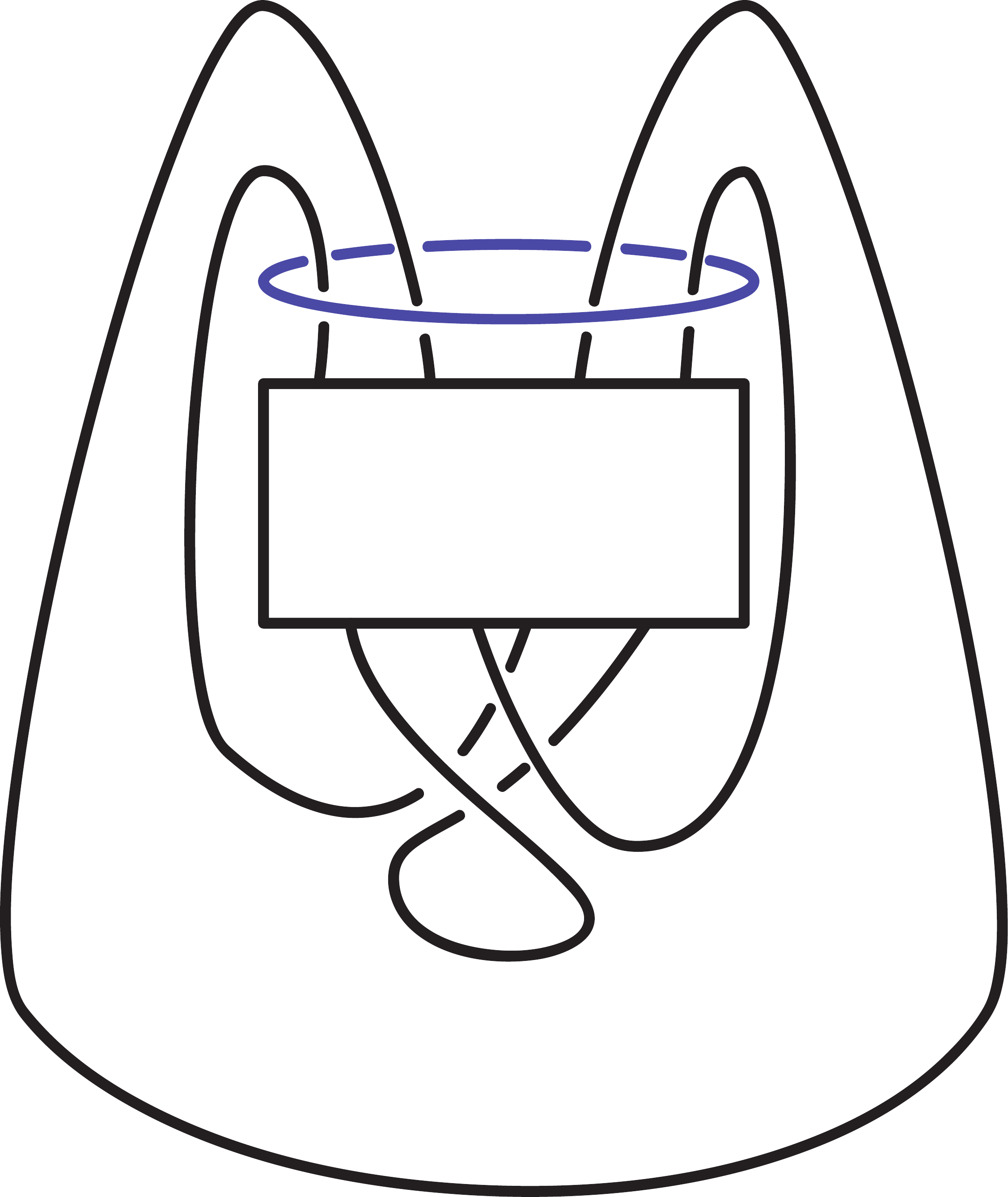}
\put(-100,200){$\eta_n$}
\put(-105,130){\huge$n$}
\caption{$K_{n}$, a knot with $n$ full twists between the bands in the box}
\label{fig:BurkeKnot}
\end{figure}
 \end{proof}
 
 Before we move onto the proof of the main theorem, Theorem \ref{thm:nonTriviality} we need a generalization of an $(n,P)$-solution known as an $(n,P)$-bordism defined in \citep{CHL11}.
 
 \begin{defn}\label{defn:nBordism}{\citep[Definition 7.11]{CHL11}}
 A compact smooth spin 4-manifold $W$ is an {\em $(n,P)$-bordism} for $\partial W$ if 
 \begin{itemize}
 \item $H_2(W;\Z)/H_2(\partial W;\Z)$ has a basis consisting of connected compact oriented surfaces $$\{L_i,D_i \mid 1 \leq i \leq r \quad L_i, D_i\quad \mbox{are embedded in}\quad W\}$$
 \item The surfaces $\{L_i,D_i\}$ have trivial normal bundles.
 \item The surfaces $\{L_i,D_i\}$ are pairwise disjoint except that for each $i$, $L_i$ intersects $D_i$ once transversely with positive sign.  
 \item For each $i$, $\pi_1(L_i) \subset \pi_1(W)_P^{(n)}$ and $\pi_1(D_i) \subset \pi_1(W)_P^{(n)}$.
 \end{itemize}
 \end{defn}
 
 Suppose that $M$ is a 3-manifold and that $\phi: \pi_1(M) \rightarrow \Gamma$ is a group homomorphism.  Cheeger and Gromov defined the $\rho$ invariant which has the $M$ and $\phi$ as its data and is denoted $\rho(M,\phi)$ \citep{ChG85}.  Since $M_K$, the zero surgery on a knot, is a 3-manifold for each group homomorphism $\phi$ we have an invariant $\rho(M_k,\phi)$ and when we say $\rho_0(K)$ we mean $\rho(M_K, Ab)$ where $Ab$ is the abelianization of the fundamental group.  Another property of the $\rho$ invariant is that if $\phi$ factors through $\phi': \pi_1(M) \rightarrow \Gamma'$ where $\Gamma'$ is a subgroup of $\Gamma$ then $\rho(M,\phi) = \rho(M,\phi')$ \citep[Prop. 5.1]{CHL11}.  There is also a universal bound on the $\rho$ invariant which we denote $C_{K}$ for a knot $K$, more specifically for any $\phi$ there exists a $C_{K}$ such that $\mid \rho(M_{K},\phi) \mid < C_{K}$ \citep[Prop 2.3]{CHL08}. 
 
We are now ready to state and prove our main theorem but we give a few comments.  First, the results in Section \ref{sec:doubOp} should be viewed in two parts.  The first is Theorem \ref{lem:Triviality}, which shows that given a sequence of polynomials then some knots $K$ that are $n$-solvable become $(n+1,P)$-solvable if they do not have appropriate torsion Alexander modules.  Theorem \ref{lem:Triviality} does not guarantee that if a knot has the appropriate torsion Alexander modules then it is nontrivial in $\F_{n}/\F_{n+1}^{P}$.  The second is Theorem \ref{thm:nonTriviality}, which shows that there exist knots with the appropriate Alexander modules which do not shift levels in the filtration for the same $P$.  
 
\begin{thm}\label{thm:nonTriviality}
Let $g(t_1,\dots,t_\ell)$ be one of polynomials from Proposition \ref{ex:infFam}.  Let $R_n$ be a knot $J_\ell$ from Lemma \ref{lem:knotFam}.  Let $T_n$ be a copy of the complement of the trivial string link with $\ell$ components embedded in $S^3 \backslash \nu(R_n)$ such that each meridian $\mu_{i,n}$ of $T_n$ is mapped to a unique $\eta_i$ from Lemma \ref{lem:knotFam}.
Let $R_{n-1}$ be a $\ell$-component string link such that 
\begin{itemize}
\item the standard closure of $R_{n-1}$ is a ribbon link from Proposition \ref{prop:ExistenceOfLink}
\item the torsion Alexander polynomial $\Delta_{R_{n-1}} = g(t_1,\dots,t_\ell)g(t_1^{-1},\dots,t_\ell^{-1})$
\end{itemize}

Let $T_{n-1}$ be an unknotted representative of the generator of $T \cA(R_{n-1})$ which exists by Proposition \ref{prop:ExistenceOfLink}.  Let $(R_{i})_{T_i}$ be robust doubling operators as in \citep[Definition, 7.2]{CHL11} for all $1 \leq i \leq n-2$, note that these are knots.  Let $P = (\Delta_{R_1},\Delta_{R_2},\dots,\Delta_{R_n})$.  Let $L_0$ be a $0-$solvable knot with $$ \mid \rho_0(M_{L_0}) \mid > \sum C_{R_i} + 2n + \ell$$ (here $C_{R_i}$ are the Cheeger-Gromov bounds on the $\rho$ invariants).  Then for knots of the form $$L_n = (R_n)_{T_n} ( (R_{n-1})_{T_{n-1}}(\dots((R_1)_{T_1} (L_0)))),$$ we have that $L_n \in \F_n / \F_{n+1}^P$ is non trivial.
\end{thm}

\begin{proof} First to simplify notation let $L_k = (R_k)_{T_k}(L_{k-1})$.  By Proposition \ref{prop:Solvable} we note that $L_n \in \F_n$.  For the sake of a contradiction suppose that $L_n$ is $(n+1,P)$-solvable, then there exists an $(n+1,P)$-solution for $M_{L_n}$, namely $V_n$.  Consider the crucial cobordism between $M_{R_k} \sqcup M_{L_{k-1}}$ and $-M_{L_{k}}$, call it $E_k$.  Let $V_{n-1} = V_n \sqcup E_n$ gluing along $M_{L_n}$, and inductively $V_{n-i} = V_{n-i+1} \sqcup E_{n-i+1}$ gluing along $M_{L_{n-i+1}}$.  Observe that $\partial V_i = M_{R_n} \sqcup M_{R_{n-1}} \sqcup \cdots \sqcup M_{R_{i+1}} \sqcup M_{L_i}$.  

We claim that $V_i$ is an $(n+1,P)$-bordism for each $i$.  We only need to verify the conditions on homology and on the fundamental group, because, by construction, $V_i$ it is a compact and spin 4-manifold.  We do this by induction showing that for $i > 0$, $H_2(V_{n-i};\Z)$ is generated by $H_2(V_n) \oplus_{k=n}^{n-i+1} H_2(E_{k}; \Z)$.  

We start with the base case $i = 1$.  The Mayer Vietoris sequence for homology implies that $H_2(V_n) \oplus H_2(E_n) \rightarrow H_2(V_{n-1}) \rightarrow H_1(M_{L_{n}}) \rightarrow H_1(V_n) \oplus H_1(E_n)$ is exact.  We see that $H_1(M_{L_n})$ injects into $H_1(E_n)$ by Lemma \ref{lem:cobordismPi} and therefore the map $H_1(M_{L_n}) \rightarrow H_1(V_n) \oplus H_1(E_n)$ is an injection.  This implies that that the map $H_2(V_n) \oplus H_2(E_n) \rightarrow H_2(V_{n-1})$ is surjective.

This finishes the base case and we move on to the inductive case.  Let $U_1 = V_{n-i+1}$ and let $U_2 = E_{n-i+1}$.  Applying Mayer Vietoris we get the exact sequence $$H_2(V_{n-i+1}) \oplus H_2(E_{n-i+1}) \rightarrow H_2(V_{n-i}) \rightarrow H_1(M_{L_{n-i+1}}) \rightarrow H_1(V_{n-i+1}) \oplus H_1(E_{n-i+1})$$.  The group $H_1(M_{L_{n-i+1}})$ injects into $H_1(E_{n-i+1})$ by Lemma \ref{lem:cobordismPi} and therefore the map $H_1(M_{L_{n-i+1}}) \rightarrow H_1(V_{n-i+1}) \oplus H_1(E_{n-i+1})$ is an injection.  This implies that $H_2(V_{n-i})$ is generated by $H_2(V_{n-i+1}) \oplus H_2(E_{n-i+1})$, which is generated by $H_2(V_n) \oplus_{k = n}^{n-i+1}H_2(E_k)$, as desired.

For each $E_i$, $H_2(E_i)$ is supported on the boundary by Lemma \ref{lem:cobordismH}.  Therefore $H_2(V_i)/ H_2(V_i,\partial V_i)$ is generated by $H_2(V_n)$ verifying the homology condition.  The fundamental group condition follows from the fact that $V_n$ is an $(n+1,P)$-solution.

Let $\Lambda_n =\pi_1(V_n)/\pi_1(V_n)^{(1)}_P$ and  $\R_n = \Q[\Lambda_n]S^{-1}_n$, where $S_n^{-1}$ is used to construct the $P$ series as in Definition \ref{defn:pSeries}.  We see that $\pi_1(V_n)^{(1)}_P = \pi_1(V_n)^{(1)}$ by Definition \ref{defn:pSeries}.  It follows that $\Lambda_n = \Z$ because $V_n$ is an $(n+1,P)$-solution for a knot.  We take homology with twisted coefficients $\R_n$ (note that in the background there is a homomorphism from $\pi_1(M_{L_n})$ and $\pi_1(V_n)$ to $\Lambda_n$ for the twisted homology).  There is a map $i_*: H_1(M_{L_n}; \R_n) \rightarrow H_1(V_n; \R_n)$ induced by the isomorphism $i_*: H_1(M_{L_n}; \Z) \rightarrow H_1(V_n; \Z)$.  The kernel of the map $i_*:H_1(M_{L_n}; \R_n) \rightarrow H_1(V_n;\R_n)$ is self perpendicular under the localized Blanchfield form by \citep[Thm 7.15]{CHL11}.

Recall that $L_n = (R_n)_{T_n}(L_{n-1})$.  We show that $(R_n)_{T_n}$ is a robust torsion doubling operator.  We see that $H_1(T_n)$ maps into $H_1(M_{L_n}; \R_n)$ since the meridians $\mu_{i,n}$ of $T_n$ map into $\pi_1(M_{R_n})^{(1)}_P$.  We see that $\blf_{\R_n} = \bigoplus \blf_{K_{i}}$ from Lemma \ref{lem:knotFam} and by construction $\blf_{\R_n}(\eta_i,\eta_i) = \blf_{K_i}(\eta_i,\eta_i) \neq 0$.  Fixing an $i$ and localizing at $\Delta_{K_i}$ we see that $\blf_{K_j} = 0$ for $j \neq i$, since the $\Delta_{K_i}$ are strongly coprime by  \citep[Example 4.10]{CHL11}.  This shows that $\blf_{\R_n}(\eta_i,\eta_j) = 0$ for $i \neq j$ and that the set $\{\eta_i\}$ is linearly independent over $\Q$ in $H_1(M_{R_n};\R_n)$.  We have shown that $(R_n)_{T_n}$ is a robust torsion doubling operator.  

Also we have shown that $H_1(T_n)$ injects into $H_1(M_{L_n};\R_n)$ since $H_1(T_n)$ injects into $H_1(M_{R_n}; \R_n)$.  Therefore there is a well defined map from $H_1(M_{L_{n-1}};\Z)$ to $H_1(V_n;\R_n)$ since the meridians of $L_{n-1}$ are identified with the meridians of $T_n$.  We can use a new coefficient system which is nontrivial on $\partial V_{n-1}$.  The current coefficient system is summarized in the Diagram \ref{fig:LambdaN} with the diagonal map being an injection.

\begin{figure}

\begin{tikzpicture}[xscale = 5,yscale = 2]
\node (A) at (-0.5,3) {$\pi_1(M_{L_{n-1}})$};
\node (B) at (-0.5,2.0) {$\dfrac{  \pi_1(V_n)^{(1)}_P   }{  \pi_1(V_n)^{(2)}_P}$  };
\node (E) at (0.5,3) {$Ab(\pi_1(M_{L_{n-1}}))$};
\path[->,font=\scriptsize,>=angle 90] 
(A) edge node[above]{} (E)
(A) edge node[right]{} (B)
(E) edge node[right]{} (B);
\end{tikzpicture}
\caption{$\Lambda_n$ Coefficients}
\label{fig:LambdaN}
\end{figure}

By Diagram \ref{fig:LambdaN} it follows that if we take $\Lambda_{n-1} = \pi_1(V_{n-1})/\pi_1(V_{n-1})^{(2)}_P$ then $i_*:\pi_1(M_{L_{n-1}}) \rightarrow \Lambda_{n-1}$ is nontrivial.  Furthermore $i_*$ factors through the abelianization $ab: \pi_1(M_{L_{n-1}}) \rightarrow \Z^\ell$, and $\Z^\ell$ embeds as a subgroup of $\Lambda_{n-1}$.  Next we construct a coefficient system on $V_{n-2}$, and then we proceed inductively.

Let $\R_{n-1} = \Q[\Lambda_{n-1}]S_{n-1}^{-1}$.  We have an induced map $i_*: H_1(M_{L_{n-1}}; \R_{n-1}) \rightarrow H_1(V_{n-1}; \R_{n-1})$.  The kernel of this map is self perpendicular under $\blf_{\R_{n-1}}$, by \citep[Thm 7.15]{CHL11}.  For the meridian $\mu_{n-1}$ of $T_{n-1}$ we have that $\blf_{\R_{n-1}}(\mu_{n-1},\mu_{n-1})$ is not equal to $0$ by the construction in Proposition \ref{prop:ExistenceOfLink}.

We now construct a coefficient system inductively.  Assume that we have a coefficient system for $V_{n-(i-1)}$, namely $\Lambda_{n-(i-1)} = \pi_1(V_{n-(i-1)})/\pi_1(V_{n-(i-1)})^{(i)}_P$ and let $\R_{n-(i-1)} = \Q[\Lambda]S_{n-(i-1)}^{-1}$.  For $\Lambda_{n-{(i-1)}}$ and $M_{L_{n-(i-1)}}$ we have that the map $i_*: \pi_1(M_{L_{n-(i-1)}}) \rightarrow \Lambda_{n-{i-1}}$ factors through the map $ab: \pi_1(M_{L_{n-(i-1)}}) \rightarrow Ab(\pi_1(M_{L_{n-(i-1)}}))$.  We also have that $Ab(\pi_1(M_{L_{n-(i-1)}}))$ injects into $\Lambda_{n-(i-1)}$.  Therefore we have an induced map $i_*: H_1(M_{L_{n-(i-1)}};\R_{n-(i-1)}) \rightarrow H_1(V_{n-(i-1)};\R_{n-(i-1)})$.  The kernel of the map $i_*$ is self perpendicular under the Blanchfield form $\blf_{\R_{n-(i-1)}}$ by \citep[Thm 7.15]{CHL11}.  Recall that $L_{n-(i-1)}$ is equal to $(R_{n-(i-1)})_{T_{n-(i-1)}}(L_{n-i})$.  By construction $\blf_{\R_{n-(i-1)}}(\mu_{n-(i-1)},\mu_{n-(i-1)})$ is not equal to $0$, where $\mu_{n-(i-1)}$ is the meridian of $T_{n-(i-1)}$.  Therefore $H_1(T_{n-(i-1)})$ embeds into $H_1(V_{n-(i-1)};\R_{n-(i-1)})$ and $H_1(M_{L_{n-i}})$ embeds into $H_1(V_{n-(i-1)};\R_{n-(i-1)})$ since the meridian of $T_{n-(i-1)}$ is identified with the meridian of $L_{n-i}$.  Taking $$\Lambda_{n-i} = \pi_1(V_{n-i})/\pi_1(V_{n-i})^{(i+1)}_P$$ we see that $\Lambda_{n-i}$ is a nontrivial coefficient system on $M_{L_{n-i}}$.  Furthermore the coefficient system $\phi: \pi_1(M_{L_{n-i}}) \rightarrow \Lambda_{n-i}$ factors through the abelianization $ab: \pi_1(M_{L_{n-i}}): \rightarrow Ab(M_{L_{n-i}}),$ and $Ab(M_{L_{n-i}})$ embeds as a subgroup of $\Lambda_{n-i}$.

By induction we have a coefficient system $\Lambda_0 = \pi_1(V_0)/ \pi_1(V_0)^{n+1}_P$ which has the property that $\phi: \pi_1(V_0) \rightarrow \Lambda_0$ is a non-trivial coefficient system, and $\Lambda_0$ is a Polytorsion Free Abelian group and $\phi(\pi_1(V_0)^{(i+1)}_P) = 1$.  We apply \citep[Thm 7.13]{CHL11} to say that the Von Neumann $\rho$ invariant $\rho(\partial V_0, \phi)$ is $0$.  

To derive a contradiction we next calculate $\rho(\partial V_0, \phi)$ using a different method.  By definition $\rho(\partial V_0, \phi) = \sigma^{(2)}(V_0) - \sigma(V_0)$ where $\sigma$ denotes the signature and $\sigma^{(2)}$ denotes the $L^2$-signature.  Then $\sigma^{(2)}(V_0) - \sigma(V_0) = \Sigma \sigma^{(2)}(E_i) - \sigma(E_i) + \sigma^{(2)}(V_n) - \sigma(V_n)$ by Novikov additivity of signatures.  Since $H_2(E_i)$ is supported on the boundary it follows that $\sigma(E_i) = 0$.  Since $V_n$ is an $(n+1,P)$ solution and since $\phi \mid_{V_n}$ also meets the hypothesis of \citep[Thm 7.15]{CHL11} it follows that $\rho(L_n,\phi \mid_{V_n}) = \sigma^{(2)}(V_n) - \sigma(V_n) = 0$.  Therefore $\rho(L_n,\phi)$ equals $\Sigma \sigma^{(2)}(E_i)$.  

On the other hand we have the following inequalities, $\mid \rho(L_0,\phi) \mid = \mid \Sigma \sigma^{(2)}(E_i) - \Sigma \rho(M_{R_i},\phi) \mid \leq \Sigma \mid \sigma^{(2)}(E_i) \mid + \Sigma C_{R_i}$, where $C_{R_i}$ are the Cheeger-Gromov bounds.  By \citep[Lemma 2.7]{CHA08} it follows that $\mid \sigma^{(2)} (E_i) \mid$ is bounded by $\beta_2(E_i)$.  Also $\rho(M_{L_0},\phi) = \rho_0(M_{L_0})$, because for $M_{L_0}$, the map $\phi$ factors through the abelianization $\Z$, and $\Z$ injects into $\Lambda_0$ by \citep[Prop. 5.1]{CHL11}.  Therefore $\mid \rho_0(M_{L_0}) \mid = \mid \rho(M_{L_0},\phi) \mid \leq \Sigma C_{R_i} + \Sigma \beta_2(E_i) \leq \Sigma C_{R_i}  + 2n + \ell < \rho_0(M_{L_0})$ which is a contradiction, therefore $L_n$ is not $(n+1,P)$-solvable.\end{proof}

One thing to note about Theorem \ref{thm:nonTriviality} is that the knot $L_0$ satisfying the condition on the $\rho$-invariant are known to exist.  This is well known because $\rho_0(L_0)$ is the integral of the Tristram-Levine signature over the circle.  Then one can compute this for a connected sum of trefoils.

\begin{cor}\label{cor:Genus}
There exist knots $K \in \mathcal{F}_2$ of genus $g$ which are not concordant to knots of genus $g'$ for any $g' < g$.  
\end{cor}

\begin{cor}\label{cor:main}
For any $m > 0 $ there exists an $m$-component string link $L$ and a ribbon knot $R$ such that for any $n$ component string link $L' \in \F_1$ with $n < m$ and any ribbon knot $R'$, the concordance classes $[R_{T}(L)]$ and $[R'_{T'}(L')]$ are distinct.
\end{cor}

\begin{proof} Let $R$ be a ribbon knot and $L$ be any string link of $m$ components as in Theorem \ref{thm:nonTriviality}, and let $L'$ be a string link of $n$ components, where $n < m$.  From Proposition \ref{ex:infFam} we see that $\Delta_L$ is strongly coprime to all polynomials in fewer variables.  Therefore $R'_{T'}(L')$ is an element of $\F^P_3$ by Theorem \ref{lem:Triviality}.  Applying Theorem \ref{thm:nonTriviality} we see that $R_{T}(L)$ is not in $\F^P_3$, and therefore is not concordant to $[R'_{T'}(T)]$.  Since $T$ was arbitrary this concludes the proof.\end{proof}

\section{Infection Curves from the Alexander Module}\label{sec:goodAModule}
One of the key steps in proving  Theorem \ref{thm:nonTriviality} is finding a curve in the Alexander module which links itself nontrivially.  This helps extend the coefficient system.  This section gives a result on finding such curves for links without having an explicit picture of the link.  If there exists a link with a special type of torsion Alexander polynomial then one can apply the same proof as in Theorem \ref{thm:nonTriviality} and obtain a nontriviality result.

Abstractly, for certain polynomials of the form $p(x_1,\dots,x_n)p(x_1^{-1},\dots,x_n^{-1})$ we can find a ribbon link $L$, with $\Delta_L = p(x_1,\dots,x_n)p(x_1^{-1},\dots,x_n^{-1})$.  We want to infect $L$ at $\eta$ (an unknotted curve in $S^3 \backslash L$) by a knot $K$ to get $L_{\eta}(K)$ which is not slice.  Then we infect a knot $R$ with $L_\eta(K)$ to try to get nontrivial elements in the knot concordance group.  Both of these operations require finding curves in $T\cA(R)$ or $T\cA(L)$ that are not trivial under inclusion.  Proposition \ref{prop:Blanch} is a tool to help with this task.  Note that by reversing orientation on the link we get a group isomorphism from $\cA(L)$ to $\cA(rL)$ where $x_i$ maps to $x_i^{-1}$.  For arbitrary $v$ in $\cA(L)$, we denote its image under this group homomorphism by $\overline{v}$.

\begin{prop}\label{prop:Blanch}
Let $L$ be a link with torsion Alexander polynomial of the form $$\Delta_L = p(x_1,\dots,x_n)p(x_1^{-1},\dots,x_n^{-1}).$$  Assume that  $p(x_1,\dots,x_n)$ is irreducible and $p(x_1,\dots,x_n),p(x_1^{-1},\dots,x_n^{-1})$ are relatively prime.  Then there exists a curve $\eta$ in $\cA(L)$ such that $\blf(\eta, \eta)$ is not $0$.
\end{prop}


\begin{proof} If $T\cA(L)$ is cyclic, that is $\langle \gamma \rangle = T \cA(L)$ for some $\gamma$.  Then $\blf(\gamma,\gamma) \neq 0$ because the Blanchfield form is nontrivial.  So we may assume $T\cA(L)$ is not cyclic.  Since the Blanchfield form is nontrivial we have there exist $e_1,$ and $e_2$ in $T\cA(L)$ such that $\blf(e_1,e_2)$ is not equal to $0$.  If for some $i$ we have that $\blf(e_i,e_i)$ is not $ 0$, then by setting $\eta = e_i$ we are done.  Assume that $\blf(e_i,e_i)$ is $0$, for $i$ equal to $1$ or $2$.  There are three possible values for $\blf(e_1,e_2)$, since $e_1$ and $e_2$ are $\Delta_L$ torsion.  The possible values are $f/\Delta_L$, $f/p$, or $f/\overline{p}$, where $f$ is an element of $K-\Lambda$.  If $\blf(e_1,e_2) = f/\Delta_L$ then we take $e_1' = \overline{p}e_1$.  We conclude that $\blf(e_1',e_2)$ equals $f/p$.  From this fact and by sesquilinearity of $\blf$, we may assume without loss of generality that $\blf(e_1,e_2)$ equals $f/p$.

Calculating $\blf(e_1+e_2,e_1+e_2)$ we see that $\blf(e_1+e_2,e_1+e_2)$ equals $\blf(e_1,e_2)+\blf(e_2,e_1)$, since $\blf(e_i,e_i)$ is $0$ by assumption.  Substituting the values we obtain the following equalities, $\blf(e_1+e_2,e_1+e_2) = f/p + \overline{f}/\overline{p} = (f\overline{p}+ \overline{f}p)/\Delta_L$.  We claim that $\blf(e_1+e_2,e_1+e_2)$ does not equal $0$.

To show this assume for the sake of contradiction that $\blf(e_1+e_2,e_1+e_2)$ is $0$.  Therefore $(f\overline{p}+ \overline{f}p)/\Delta = g$ is an element of $\Lambda$, which means that $f\overline{p}+ \overline{f}p$ equals $g\Delta_L$.  It follows that $p$ must divide $f$ because $p$ is prime, $p$ is relatively prime to $\overline{p}$, $p$ divides $g\Delta_L$, and $\overline{f}p$.  This is a contradiction because $p$ dividing $f$ is the same as $\blf(e_1,e_2) = 0$ but $\blf(e_1,e_2) \neq 0$ by assumption.\end{proof}

\bibliographystyle{plainnat}

\end{document}